\documentclass[11pt,a4paper,english]{article}
\usepackage{authblk}
\usepackage{upgreek}
\usepackage{epsfig}
\usepackage{mathtools}
\usepackage{mathrsfs}
\usepackage{amsthm}  
\usepackage{amssymb} 
\usepackage{bbm}
\usepackage[nomath,notext]{kpfonts} 
\usepackage{enumitem}
\usepackage{bigints}
\usepackage{color}
\usepackage{float}
\usepackage{ifthen}
\usepackage{bm}
\usepackage{fix-cm} 
\usepackage{biblatex} 
\DeclareMathAlphabet{\mathsfbd}{T1}{\sfdefault}{\bfdefault}{\itdefault}
\SetMathAlphabet{\mathsfbd}{bold}{T1}{\sfdefault}{\bfdefault}{\itdefault}

\DeclareMathAlphabet{\mathsfbdit}{T1}{\sfdefault}{\bfdefault}{\itdefault}
\SetMathAlphabet{\mathsfbdit}{bold}{T1}{\sfdefault}{\bfdefault}{\itdefault}

\DeclareMathAlphabet{\mathsfit}{T1}{\sfdefault}{\mddefault}{\sldefault}
\SetMathAlphabet{\mathsfit}{bold}{T1}{\sfdefault}{\bfdefault}{\sldefault}

\definecolor{applegreen}{rgb}{0.55, 0.71, 0.0}
\definecolor{ao}{rgb}{0.0, 0.5, 0.0}

\newtheorem{theorem}{Theorem}[section]
\newtheorem{corollary}[theorem]{Corollary}
\newtheorem{lemma}[theorem]{Lemma}

\newtheorem{proposition}[theorem]{Proposition}
\newtheorem*{remark}{Remark}

\newcommand{\EMc}{\ensuremath{\Upupsilon}} 
\newcommand{\Betaf}{\mathsf{B}}

\newcommand{\N}{\ensuremath{\mathbb{N}}}  

\newcommand{\restr}[2]{{
  \left.\kern-\nulldelimiterspace 
  #1 
  \vphantom{\big|} 
  \right|_{#2} 
  }}

\renewcommand{\P}{\ensuremath{\mathbbm{P}}} 
\newcommand{\E}{\ensuremath{\mathbbm{E}}}  
\newcommand{\Var}[1]{\ensuremath{\mathbbm{Var}\(#1\)}} 

\newcommand{\stp}[1]{\(#1\)_{t\geq 0}} 
\newcommand{\given}[1]{\ensuremath{\vert #1}} 

\newcommand{\eps}{\epsilon}

\newcommand{\dif}[1]{\hspace{3pt}d#1}
\newcommand{\der}[2]{\frac{d}{d#2}#1} 
\newcommand{\bO}[1]{\mathcal{O}\(#1\)} 
\newcommand{\lo}[1]{o\(#1\)} 

\newcommand{\Ind}[1]{
\ensuremath{\mathbbm{1}_{#1}}
} 
\renewcommand{\S}[1]{ 
\ifx&#1&
\ensuremath{\mathcal{S}}
\else
\ensuremath{\mathcal{S}}\(#1\)
\fi
}
\newcommand{\Bo}[1]{  
\ifx&#1&
\ensuremath{\mathscr{B}}
\else
\ensuremath{\mathscr{B}}(#1)
\fi
}
\newcommand{\Lp}[1]{\mathsf{L}^{#1}} 

\DeclarePairedDelimiterX{\IP}[2]{\langle}{\rangle}{#1, #2} 
\newcommand{\norm}[1]{\ensuremath{\left\lVert #1 \right\rVert}} 
\newcommand{\stf}[1]{ 
\ifx&#1&
\ensuremath{S}
\else
\ensuremath{S\(#1\)}
\fi
} 
\newcommand{\Id}{\ensuremath{\bm{I}}} 
\newcommand{\Op}[1]{\ensuremath{\bm{#1}}} 
\newcommand{\mat}[1]{\Op{#1}}

\newcommand{\PS}[1]{\mathscr{P}_{#1}} 

\DeclareMathOperator{\coag}{Coag} 

\newcommand{\ze}{\mathbf{0}}

\renewcommand{\(}{\left(}
\renewcommand{\)}{\right)}
\renewcommand{\[}{\left[}
\renewcommand{\]}{\right]}
\newcommand{\lbr}{\left\{} 
\newcommand{\rbr}{\right\}} 
\newcommand{\abs}[1]{\left\lvert #1 \right\rvert}
\newcommand{\card}[1]{\##1}

\DeclarePairedDelimiter\floor{\lfloor}{\rfloor}
\newcommand{\dPi}{\tilde{\Pi}} 
\newcommand{\dPiN}{\dPi^{(N)}}
\newcommand{\dpi}{\tilde\pi}
\newcommand{\cp}{\varPi} 

\newcommand{\cpNn}{\cp^{(N,n)}}

\newcommand{\rpvN}{\eta^{(N)}}
\newcommand{\rpv}{\bm\eta^{(N)}}

\newcommand{\sbpN}{s^{(N)}} 

\newcommand{\osN}{\bm\xi^{(N)}} 
\newcommand{\osi}[1]{\xi_{#1}^{(N)}}
 
\newcommand{\PN}{\mathbf{P}} 
 
\newcommand{\perm}{\sigma} 

\renewcommand{\Var}[1]{\mathbbm{Var}\(#1\)}
\newcommand{\ppwr}{\gamma}
\renewcommand{\card}[1]{\abs{#1}}

\begin{document}
\title{Exchangeable coalescents beyond the Cannings class}
\author[1]{Arno Siri-J\'egousse}
\author[2]{Alejandro H. Wences}
\affil[1]{IIMAS, Universidad Nacional Aut\'onoma de M\'exico, México}
\affil[2]{LAAS-CNRS, Université de Toulouse, CNRS, Toulouse, France}


\maketitle
\begin{abstract}
We propose a general  framework for the study of the genealogy of neutral discrete-time populations. 
We remove the standard
assumption of exchangeability of offspring distributions {appearing in Cannings models}, and
replace it by a {less restrictive condition of} non-heritability of reproductive success. 
We provide a general criterion for the weak convergence of their genealogies 
to $\Xi$-coalescents, and apply it to a  simple parametrization of our scenario (which, under mild
conditions, we also prove
to essentially include the general case).
We provide examples for such populations, including models with highly-asymmetric offspring distributions
and populations undergoing random but recurrent bottlenecks. 
Finally we study the limit genealogy of
a new exponential model which, as previously shown for related models and 
in spite of its built-in (fitness) inheritance mechanism, 
can be brought into our setting.
\end{abstract}

\section{Introduction}

Coalescent processes appear as the natural limits that describe
the genealogy of evolving population models. In fact, much of the research 
on population models focuses on establishing this connection, often as weak limits under 
appropriate time scales and/or time changes. 
A primordial class of population models are Cannings models \cite{Cannings74,Cannings75}. These 
are able to accommodate most biological premises that fall under
the assumption of neutral evolution. The key assumption of these models is that the offspring distribution of 
the parents in any generation
should be {exchangeable (invariant under permutations)};
 in particular all the parents have the same reproductive success, or fitness, in distribution. 
Due to their wide applicability and manageability, Cannings models
are among the most studied population evolution models.
The genealogy of Cannings models was described in whole generality, 
and in terms of coalescent processes, in the foundational work of \cite{MohleSagitov2001}. 
After this, a wide literature was developed on the study of limit genealogies of evolution models.
Some are direct applications e.g. \cite{HuilletMohle2021}, or deal with a priori more complex models that can be reduced to Cannings models \cite{Schweinsberg2003, CortinesMallein2017}.
Others are modifications, such as models with varying population sizes  \cite{Mohle2002, KajKrone2003, Freund2020} or diploid reproduction \cite{MohleSagitov2003,BirknerLiuSturm2018}, using heuristics close to those of \cite{MohleSagitov2001}. 
Finally some works use different techniques such as duality, e.g. in the case of populations suffering recurrent bottlenecks \cite{CasanovaMiroJegousse2020} or overlapping generations (Casanova et al.  preprint arXiv:2210.05819).

Here we study a generalization of Cannings 
models, the so-called Asymmetric Cannings (AC) models, that, nonetheless, still fall under
the umbrella of neutral evolution. Our results provide weaker conditions for the main result in \cite{MohleSagitov2001}, expanding its applicability and providing a better insight into neutral evolution from a mathematical perspective. Having an accurate understanding of the neutral case is paramount for the theory and applications where the effect of natural selection is studied. One of our main contributions is the change in
focus in neutral models from the exchangeability (symmetry) condition on 
offspring distributions, to the weaker condition of non-heritability of reproductive success. 
It is in fact from the latter property, and not the former, 
from which neutrality stems; both, from biological 
(heritability is a necessary condition for natural selection to occur, whereas asymmetric reproductive success is not sufficient) and mathematical perspectives. 

Indeed, as argued in \cite{MohleSagitov2001}, the assumed exchangeability of the offspring distribution in Cannings models
ensures that the partition-valued ancestry process evolves, in distribution, 
through sequential ``putting balls into boxes'' steps.
In fact it is from this last property from which the Markovicity and exchangeability of the ancestry process stem. 
In this work we thus remove the exchangeability condition on the offspring distribution of Cannings models, 
and instead model neutral evolution by removing any form of inheritance, letting children choose parents
in an exchangeable way. In other words, we allow individuals 
in any given generation to have different offspring distributions,
but impose that these differences in reproductive success
stem from non-heritable traits. An example of this would be a population whose individuals randomly explore different 
ecological niches in every new generation, resulting in non-heritable asymmetries in reproductive success. Formally, AC models consist of a population with constant size $N$ evolving in discrete time.
Letting $\osN(t)\coloneqq\(\osi{1}(t),\cdots,\osi{N}(t)\)$ be the offspring sizes of each parent in generation 
$t\in\N$, the
 framework makes the following assumptions:
\begin{enumerate}
\item {\bf Static environment.} The sequence of vectors of offspring sizes $\stp{\osN(t)}$ is i.i.d..
\item {\bf Non-heritability.} Children choose parents in an exchangeable way; i.e. the random vector of indices indicating the parent of each child in the previous generation must be exchangeable (and the number of repetitions of each index should of course comply with $\osN$).
\end{enumerate}

As a special case, we focus on a parametrization of this scenario, the class of Asymmetric Wright-Fisher (AWF) models, in which
the offspring distribution is assumed to be multinomial. 
The key difference between multinomial and general offspring distributions (within our framework), 
is that in AWF models children choose parents
via a sampling with replacement, whereas in the general AC case the sampling is done without replacement. Taking this 
distinction into
account one can easily recast our results from AWF models to AC models. 
In the
AWF framework, children of each generation choose their parent according to independent copies of a random probability vector $\rpv=(\rpvN_1,\dots,\rpvN_N)$, and
the distribution 
 of the offspring number of individuals at a given generation is multinomial with parameters $(N,\rpv)$.

 {In this work we provide a general criterion for the weak convergence of 
discrete-time coalescent processes to $\Xi$-coalescents as $N\to\infty$; which in particular applies to the
genealogy of AC and AWF populations. 
This in fact follows by a careful study of the corresponding arguments
in \cite{MohleSagitov2001} for Cannings populations. 
In the case of AC models, one may argue that any vector of offspring sizes $(\osi{1},\dots,\osi{N})$ becomes exchangeable once we
 randomly permute its entries, thus effectively rendering any model with non-heritability, and in fact more general models, into a Cannings model. In other words, the distribution of the genealogy of the AC model is equal to that of 
the Cannings model with offspring sizes $(\osi{\sigma(1)},\dots,\osi{\sigma(N)})$ where $\sigma$ is a uniform permutation of $[N]$.
This loss of information on the asymmetric reproductive success of the population will, however, over-complicate the
study of its genealogy. We are able to 
avoid this unnecessary step and harness the asymmetry of the reproduction laws to more easily derive the limit genealogy of the population.

We also show that, in most cases, the difference between sampling with and without replacement becomes negligible, making 
the dynamics of AC and AWF models essentially equivalent. This is a general
result in the same spirit as \cite{HuilletMohle2021}, where the authors study the sampling with replacement version of the model in \cite{Schweinsberg2003}, see also \cite{GPS} for a multinomial version of Cannings models with seed bank effects.}

The new class of asymmetric models and the techniques developed in this paper have some straightforward applications to population models that do not satisfy Cannings assumption of exchangeability.
 A first example is given by the celebrated model developed in \cite{EldonWakeley2006} for populations with skewed offspring distributions where only one individual per generation reproduces. 
Another example of asymmetric reproduction is given by (a slight modification of) the model in \cite{CasanovaMiroJegousse2020} for populations suffering recurrent bottlenecks, for which a whole generation can be produced by a small amount of individuals from the previous one.
 
A more interesting consequence of our method is its applicability to a certain class of non-neutral models.
 Exponential models  \cite{BrunetDerrida97} provide a surprising example 
where mutation and selection are incorporated and, however, the built-in inheritance mechanism in fact disappears due to
mathematical transformations (mainly Poisson point processes identities). 
These models can therefore be studied under
 the AWF framework. The exponential model we are considering here is an extension that combines those of \cite{CortinesMallein2018,CortinesMallein2017}.
 As for most models incorporating strong selective pressure,
their limit genealogy is described by the coalescent first introduced by \cite{BolthausenSznitman98}. 
The Bolthausen-Sznitman coalescent 
 has been proposed as a null model for the genealogies of rapidly adapting populations  (see e.g. \cite{NeherHallatschek2013}). Many examples in the literature have established this link, such as those in \cite{NeherHallatschek2013, Schweinsberg2017-II, CortinesMallein2018, CortinesMallein2017, BerestyckiBerestyckiSchweinsberg2013, SchertzerWences2023},
to mention a few.
To the best of our knowledge, our model is the first example in which the convergence to the Bolthausen-Sznitman coalescent holds in a time scale other than logarithmic on the size $N$ of the population.

 This paper is structured as follows.
 In section  \ref{sec:multinomialOffsp}, we establish the main convergence results for the genealogies of both AC and AWF models, but also establish a correspondence so that more tractable results for AWF models can be used in the AC case.
In section \ref{seq:largeRepEvents}, we illustrate how we can exploit the asymmetry
when the limit genealogy is a $\Lambda$-coalescent, simplifying the proofs and generalizing similar results 
found in the literature. 
The same techniques can also be used for the general
$\Xi$-coalescent case; for simplicity and space constraints we do not include such development here.
Section \ref{neutralmodels} is devoted to the application of our main results to neutral asymmetric models of skewed offspring and bottlenecks. 
In section  \ref{sec:multPDcase} we focus on a particular example of an AWF model where the family frequencies are obtained from a Poisson-Dirichlet partition, and establish a relation with Beta-coalescents in this case. Thanks to this new result, we will be able to study a new class of {exponential models}.

\section{The genealogy of models with non-heritability} \label{sec:multinomialOffsp}
\subsection{Weak convergence of coalescent processes}
In this section we present our main weak convergence theorem for general discrete-time coalescent processes. It is in fact a re-framing of the main theorem in \cite{MohleSagitov2001} which, nonetheless, uncovers its generality and its possible application to population models
beyond the Cannings class.\par
Let us first introduce some notation and recall some terminology from \cite{Bertoin2006}. Let $\PS{n}$ (resp. $\PS{\infty}$) be the set of partitions of $[n]\coloneq \{1,\cdots,n\}$ (resp. $\N$)
endowed with their usual topology 
(Lemma 2.6 in \cite{Bertoin2006}),
and $\ze_n\in\PS{n}$ (resp. $\ze_\infty$) be the partition of singletons. 
For any partition $\pi$, denote its blocks (its composing elements), ordered by their least elements, by
 $\pi_1,\pi_2,\cdots$. More precisely, 
\begin{itemize}
\item $\pi_i\in\pi$ for all $i\in\N$
\item $\pi =\bigcup_{i=1}^{\infty} \pi_i$
\item $\min\{k: k\in \pi_j\} \geq \min\{k: k\in\pi_i\}$ for all $j\geq i $.
\end{itemize}
Denote the cardinality of any set $A$ by $\card{A}$. Write $\coag(\cdot,\cdot)$ for
the usual coagulation operator on $\PS{n}\times\PS{n}$ (Definition 4.2 in \cite{Bertoin2006}): 
if $\pi'\in\PS{m}$ and $\pi\in\PS{n}$ are such that
$\card{\pi} \leq m\leq \infty$ we say that the pair $(\pi,\pi')$ is admissible, and define the coagulation 
of $\pi$ by $\pi'$ as
\begin{equation*}
\coag(\pi,\pi')=(\hat{\pi}_1,\hat{\pi}_2,\cdots),
\end{equation*}
where 
\begin{equation*}
\hat{\pi}_k \coloneqq \bigcup_{j\in \pi'_k} \pi_j.
\end{equation*}
Of course $\pi_j$ is set to $\emptyset$ if $j>\abs{\pi}$. \par
Consider a discrete-time coalescent process $\left(\cpNn_t\right)_{t\in\N}$ 
constructed as
\begin{equation}\label{eq:multiGenealogyDef}
\cpNn_t\coloneqq \begin{cases}
\ze_n, & t=0,\\
\coag\(\cpNn_{t-1},\dPiN_t\), & t > 0,
\end{cases}
\end{equation}
where the family of coagulation increments $\(\dPiN_h\)_{h\in\N}$ is i.i.d. and each partition $\dPiN_h$ is exchangeable (i.e., its distribution is invariant under any random permutation of the elements in its blocks, see e.g. section 2.3.2 \cite{Bertoin2006}).
We are interested in the weak limit as $N\to\infty$ of $\stp{\cpNn_{\lfloor t/c_N\rfloor}}$. 
The correct time scale,
 as is for Cannings models, 
is given by the probability $c_N$ that two randomly chosen individuals have the same parent in the previous generation. Thanks to exchangeability, this probability can be written in terms of the indices $1$ and $2$  as
\begin{equation}\label{eq:generalcN}
c_N= \P\left(1 \text{ and }2\text{ belong to the same block of }\dPiN_h\right).
\end{equation}\par
We now introduce some additional notation. Let $\PS{[0,1]}$ be the space of mass partitions 
$$
\PS{[0,1]}\coloneqq \{(\rho_1,\rho_2,\cdots)\colon \rho_1\geq \rho_2\geq\cdots, \text{ and }\sum_{i=1}^\infty \rho_i \leq 1\}
$$
endowed with the $\Lp{\infty}$ topology,
and consider a $\Xi$-coagulation measure on $\PS{[0,1]}$ (see Section 4.2.4 in \cite{Bertoin2006}). Let
 $\mat Q^{(n)}_{\Xi}$ be the coagulation matrix of the 
$\Xi$-coalescent with values on $\PS{n}$; if $(\pi,\pi')$ is an admissible pair of partitions,  $\mat Q^{(n)}_{\Xi}(\pi,\coag(\pi,\pi'))$ gives the rate at which a $\Xi$-coalescent jumps from $ \pi$ to 
$\coag(\pi,\pi')$ . If $I\equiv I_{\pi'}$ is the set of indices $i$ such that $\card{\pi'_i}=1$, and for $\bm\rho\in\PS{[0,1]}$ we write $\rho_0=1-\sum_{i=1}^\infty  \rho_i$, then these rates are given by
\begin{equation}\label{eq:QratesFormula}
\mat Q^{(n)}_{\Xi}(\pi,\coag(\pi,\pi')) = \int_{\PS{[0,1]}}  \sum_{\substack{i_1,\cdots,i_{\card{\pi'}}> 0\\ \text{all distinct}}} \prod_{j\in I} (\rho_{i_j}+\rho_0)    \prod_{j\not\in I}\rho_{i_j}^{\card{\pi'_{j}}}  \frac{\Xi(d\bm \rho)}{\sum_{i=1}^\infty \rho_i^2}.
\end{equation}
(see also e.g.  Equation (11) in \cite{Schweinsberg2000}).
The above expression is a consequence of the well-known paint-box construction (see e.g. section 2.3.2 in \cite{Bertoin2006}) of random exchangeable partitions of $[N]$ from a mass partition $\bm \rho$: one partitions the unit interval into consecutive non-overlapping sub-intervals of sizes $\rho_1,\rho_2,\cdots,$ respectively and, letting $(U_i)_{1\leq i\leq N}$ be i.i.d. standard uniform random variables,
one constructs a partition of $[N]$ by putting together in the same block all those indices $i$ such that $U_i$ falls on the same interval of type $[\sum_{j=1}^k \rho_j, \sum_{j=1}^{k+1} \rho_j)$; whereas all those indices whose $U_i$ falls on the
interval $[1-\rho_0,1)$ become singleton blocks.  In equation \eqref{eq:QratesFormula},  it is the coagulation increment $\pi'$ that is paint-box constructed from a mass partition $\bm\rho$ that is chosen at rate $ \frac{\Xi(d\bm \rho)}{\sum_{i=1}^\infty \rho_i^2}$. Note that, the random
variables $U_i$ being independent, and conditional on the value of $\bm\rho$, the sum of products appearing in \eqref{eq:QratesFormula} 
simply computes the probability that the paint-box construction directed by $\bm\rho$ results in the partition $\pi'$.
Finally, for any coagulation matrix $\mat Q^{(n)}$ operating on $\PS{n}$, we use the abbreviation
$\mat Q^{(n)}(\dpi)\coloneqq\mat Q^{(n)}(\ze_n,\dpi)$. \par
The following result was previously stated and proved in \cite{MohleSagitov2001} 
for the genealogies of Cannings models. 
Here we  emphasize that the same statement, and in fact the same proofs, remain
valid (verbatim) in the context of the more general 
discrete-time coalescent processes defined as in \eqref{eq:multiGenealogyDef}. We include the corresponding argument for completion.
\begin{theorem}\label{th:multWeakConv}
Fix $n\in\N$. Let 
$$
\mat P^{(N,n)}({\pi, \coag(\pi,\tilde\pi)})\equiv \mat P^{(N,n)}(\tilde\pi)  = \P(\dPiN_h=\dpi)
$$ 
be the transition matrix of a discrete-time coalescent as defined in \eqref{eq:multiGenealogyDef}. 
Also, let $\mat Q^{(n)}_{\Xi}$ be a coagulation matrix associated to a measure $\Xi$ on $\PS{[0,1]}$.
 Assume that 
\begin{equation}\label{eq:TaylorMassPartition}
\mat P^{(N,n)}(\dpi) = \begin{cases}
c_N\mat Q^{(n)}_{\Xi}(\dpi) + \lo{c_N} &\text{ if }\dpi\in \PS{n}\setminus \{\ze_n\}\\
1-c_N\sum_{\pi\in \PS{n}\setminus\ze_n} \mat Q^{(n)}_{\Xi}(\pi) + \lo{c_N}& \text{ if }\dpi=\ze_n.
\end{cases} 
\end{equation}
{i)}  If $c_N\to c>0$ as $N\to\infty$, then $\left(\cpNn_{t}\right)_{t\in\N}$ converges weakly in the product topology for $(\PS{n})^\N$  
 to a
Markov chain with initial state $\ze_n$ and transition matrix $\Id + c\mat Q^{(n)}_{\Xi}$.
\\
{ii)} If $c_N\to 0$ as $N\to\infty$, then $\stp{\cpNn_{\lfloor t/c_N\rfloor}}$ converges weakly in the Skorohod space $D([0,\infty), \PS{n})$ to
the $\Xi$-coalescent with initial state $\ze_n$.\par
\end{theorem}
\begin{proof}
To prove $i)$ note that \eqref{eq:TaylorMassPartition} plus the condition $c_N\to c>0$ imply $\lim_{N\to\infty}\mat P^{(N,n)} = \Id + c\mat Q^{(n)}_{\Xi}$ which
in turn gives convergence of the finite dimensional distributions. The latter is equivalent to weak convergence on $\PS{n}^\N$ 
(see Example 2.6 in \cite{Billingsley99}).\par
Let us now prove $ii)$. Given that both $\stp{\cpNn_{\[t/c_N\]}}$ and the $\Xi$-coalescent are Feller processes, we need only prove convergence
of their semigroups (Theorem 2.5 in \cite{EthierKurtz86}). By means of the equality
\begin{align*}
\(\mat P^{(N,n)}\)^{k} - &\(\Id + c_N\mat Q^{(n)}_{\Xi}\)^{k}\\ 
&= \sum_{i=1}^k \(\mat P^{(N,n)}\)^{i-1}(\mat P^{(N,n)} - \(\Id + c_n\mat Q^{(n)}_{\Xi}\))\(\Id + c_N\mat Q^{(n)}_{\Xi}\)^{k-i}
\end{align*} (which follows by telescoping the sum on the right hand side), we compute, for every $t>0$,
\begin{align*}
&\norm{\left(\mat P^{(N,n)}\right)^{\[\frac{t}{c_N}\]} - (\Id + c_N\mat Q^{(n)}_{\Xi})^{\[\frac{t}{c_N}\]}}\\
&\leq \[\frac{t}{c_N}\] \norm{\mat P^{(N,n)} - (\Id + c_N\mat Q^{(n)}_{\Xi})} (1 + \norm{\PN - (\Id + c_N\mat Q^{(n)}_{\Xi})})^{\[\frac{t}{c_N}\]}\\
&\leq \[\frac{t}{c_N}\] \norm{\mat P^{(N,n)}  - (\Id + c_N\mat Q^{(n)}_{\Xi})} 
\exp\lbr \[\frac{t}{c_N}\] \(\lo{c_N} + \lo{c_N^2}  \)\rbr\\
&\overset{N\to \infty}{\longrightarrow} 0;
\end{align*}
where $\norm{\cdot}$ refers to the usual operator norm, and  we have used \eqref{eq:TaylorMassPartition}. 
Therefore
$$
\lim_{N\to\infty} \left(\mat P^{(N,n)}\right)^{\[\frac{t}{c_N}\]} = \lim_{N\to\infty} \left(\Id + c_N\mat Q^{(n)}_{\Xi}\right)^{\[\frac{t}{c_N}\]} = e^{t\mat Q^{(n)}_{\Xi}}.
$$
\end{proof}

\subsection{The Asymmetric Cannings framework}\label{sec:AC}

Consider a random population satisfying the static environment and non-heritability conditions of the AC model described above.
To construct its genealogy, set a time horizon $T>0$ and for $0\leq t \leq T$ let $A^{(N)}_t$ be the random partition of $[N]$ resulting from grouping all the individuals
at time $T$ according to their common ancestors at time $T-t$. Similarly, for $n\leq N$, let $(A^{(N,n)})_{0\leq t\leq T}$ be the genealogy of a random sample of $n$ individuals from the population at time $T$ constructed in the same way. The {static environment} and {non-heritability} conditions in AC models allow us to study the genealogy $(A^{(N,n)}_t)_{0\leq t\leq T}$
through a coalescent process constructed as in \eqref{eq:multiGenealogyDef}. Indeed,
the distribution of the coagulation increments $\dPiN_h$ should simply be set to the distribution of the random partition in the AC model resulting from grouping the
first $n$ individuals (or, by exchangeability, any random sample of $n$ individuals)
in generation $T-h+1$ according to their parents in the previous generation $T-h$.
Using the non-heritability condition one easily sees that the coagulation increments  $\dPiN_h$ do indeed coincide in distribution with those of $(A^{(N,n)}_t)_{0\leq t\leq T}$. On the other hand, the static environment condition
ensures that the sequence $\dPiN_h$ is i.i.d., and thus also the Markovicity of the process $(A^{(N,n)}_t)_{0\leq t\leq T}$. Putting both observations together we see that $(\cpNn_{t})_{0\leq t\leq T}$ and $(A^{(N,n)}_t)_{0\leq t\leq T}$ coincide in distribution. Furthermore, this allows us to take $T=\infty$
in  $(\cpNn_{t})_{0\leq t\leq T}$
to describe (in distribution) the genealogy of the AC model as the time horizon grows to infinity.\par

In the AC framework the matrix of transition probabilities $\mat P^{(N,n)}$ can be expressed as
\begin{equation*}
 \mat P^{(N,n)}(\tilde\pi) =  
\E\[\sum_{\substack{i_1,\dots,i_{\card{\tilde\pi}} \\ \text{all distinct}}}
 \frac{\(\osi{i_1}\)_{b_1} \cdots \(\osi{i_{\card{\tilde\pi}}}\)_{b_{\card{\tilde\pi}}}}
      {(N)_{n}}
\]
\end{equation*}
where for  $a\in\mathbb{R}$ and $b\in\N$ we have written $(a)_b= a(a-1)\cdots(a-b+1)$, and $b_k=\card{\dpi_k}$.
In particular, the probability $c_N$ that two randomly chosen individuals share the same parent in the previous generation is given by
\begin{equation}\label{eq:cN}
c_N= \frac{1}{N(N-1)}\sum_{i=1}^N \E\[\osi{i}\(\osi{i}-1\)\].\end{equation}
Plugging these into \eqref{eq:TaylorMassPartition}, Theorem \ref{th:multWeakConv} yields a general criterion for the convergence of the genealogy of AC populations.

\subsection{The Asymmetric Wright-Fisher framework}\label{sec:AWF}
Let us now focus on AWF models with weight vector $\rpv$.
Grouping children by the parent they choose gives in fact the coagulation
increment $\dPiN_h$ appearing in \eqref{eq:multiGenealogyDef}. Indeed, grouping children in this way is precisely the well-known paint-box construction (see e.g. section 2.3.2 in \cite{Bertoin2006}) of exchangeable random partitions
directed by the mass partition $\rpv$. By a simple ``putting balls into boxes'' argument analogous to that leading to \eqref{eq:QratesFormula} above, the transition matrix $\mat P^{(N,n)}$ of $\left(\cpNn_t\right)_{t\in\N}$ in given in this case   by
\begin{equation}\label{PNpaintbox}
 \mat P^{(N,n)}(\tilde\pi) =  
\E\[\sum_{\substack{i_1,\dots,i_{\card{\tilde\pi}} \\ \text{all distinct}}}
 \(\rpvN_{i_1}\)^{b_1} \cdots \(\rpvN_{i_{\card{\tilde\pi}}}\)^{b_{\card{\dpi}}} 
\]
\end{equation}
where $b_i\coloneqq\card{\dpi_i}$, $1\leq i\leq \card{\dpi}$. 
In particular, in 
the AWF setting, the probability 
that two randomly chosen individuals have the same parent in the previous generation is now
\begin{equation}\label{eq:multcN}
c_N = \E\[\sum_{i=1}^N \(\rpvN_i\)^2\].
\end{equation}
Plugging these into \eqref{eq:TaylorMassPartition}, Theorem \ref{th:multWeakConv} yields a general criterion for the convergence of the genealogy of AWF populations. \par

We now give a characterization of \eqref{eq:TaylorMassPartition} for AWF populations, in the same spirit as is done in \cite{MohleSagitov2001} for exchangeable offspring distributions.
{Our technique, and analogous results, should apply to more general cases.}
Let us introduce further notation. We call $(s_1,s_2,\dots)$ a 
size-biased reordering of a random probability vector $\rpv=\(\rpvN_1,\dots,\rpvN_N\)$ if, conditional on $\rpv$
and letting $K=\sum_{j=1}^N \Ind{\rpvN_j>0}$  be the number of non-zero entries of $\rpv$, 
the vector $(s_1,\dots,s_{K})$ is a size-biased reordering of the non-zero entries of $\rpv$ in the usual sense (see e.g. \cite{PitmanYor97}, or equation \eqref{eq:defSizeBiased} below),
and $s_j= 0$ for $j>K$. 

\begin{theorem}\label{th:multWeakConv2}
For an AWF population, the equality \eqref{eq:TaylorMassPartition}
holds for some $\Xi$-coalescent and $c_N\to0$, i.e. case ii) holds, if and only if the limits
\begin{align}\label{eq:multLimitCoagProbs}
&\phi_j(b_1,\dots,b_j)\nonumber\\=& \lim_{N\to\infty} \frac{1}{c_N}\E\[\prod_{i=1}^{j} \(\sbpN_i\)^{b_i-1}\prod_{i=1}^{j-1} \(1-\sum_{k=1}^i\sbpN_k\) \] 
\end{align}
 exist for all $j\in\N$ and $b_1\geq \dots b_j\geq2$, where $\(\sbpN_1,\sbpN_2,\dots\)$ is a size-biased reordering of $\rpv$ .
\end{theorem}

\begin{proof}
Recall \eqref{PNpaintbox}. Observe that if $\rpvN_{i}=0$ then the corresponding terms in the right hand side are zero. Also, 
for a size-biased reordering of $\rpv$ and 
for 
distinct indices $i_1,\dots,i_b$ such that $\eta_{i_j}>0$ for all $1\leq j\leq b$, we have 
\begin{align}\label{eq:defSizeBiased}
&\P\(\sbpN_k=\rpvN_{i_k}, 1\leq k\leq b\vert \rpv\)\nonumber\\
&=\prod_{k=1}^b \P\(\sbpN_k=\rpvN_{i_k} \vert \rpv, \sbpN_1, \cdots, \sbpN_{k-1}\)\nonumber\\
&=\prod_{k=1}^b \frac{\rpvN_{i_k}}{1-\sum_{j=1}^{k-1}\rpvN_{i_j}}.
\end{align}
Thus, writing $S^{(N)}_i=\sbpN_1+\dots+\sbpN_i$, and $(b_1,\cdots,b_{\card \dpi})$ for the sizes of the blocks of $\dpi$ arranged in non-increasing order,

\begin{align*}
&\mat P^{(N,n)}({\pi,\coag(\pi, \tilde\pi)})\\&=\sum_{ \substack{i_1,\dots,i_{\card{\tilde\pi}} \\ \text{all distinct}}}
 \E\[
     \(\rpvN_{i_1}\Ind{\rpvN_{i_1}>0}\)^{b_1} \cdots  \(\rpvN_{i_{\card{\dpi}}}\Ind{\rpvN_{i_{\card{\dpi}}}>0}\)^{b_{\card{\dpi}}} \]\\
&=\sum_{ \substack{i_1,\dots,i_{\card{\tilde\pi}} \\ \text{all distinct}}}
\E\[
    \prod_{k=1}^{\card{\dpi}} \frac{\rpvN_{i_k}\Ind{\rpvN_{i_k}>0}}{1-\sum_{j=1}^{k-1}\rpvN_{i_j}} (1-\sum_{j=1}^{k-1}\rpvN_{i_j})\(\rpvN_{i_k}\)^{b_k-1} \]\\
&=\E\[\sum_{ \substack{i_1,\dots,i_{\card{\tilde\pi}} \\ \text{all distinct}}}
             \P\(\sbpN_k=\rpvN_{i_k}, k \leq \card{\tilde\pi} \vert \rpv\) \prod_{i=1}^{\card\dpi}
(1-\sum_{j=1}^{k-1}\rpvN_{i_j})\(\rpvN_{i_k}\)^{b_k-1} \]\\
&=\E\[\E\[\prod_{i=1}^{\card\dpi}\(\sbpN_{i}\)^{b_i-1}  \prod_{i=1}^{\card{\tilde\pi}-1} (1-S^{(N)}_i); s^{(N)}_1>0,  \dots, s^{(N)}_{\card{\tilde\pi}}>0 \vert \rpv \]\]\\
&=\E\[\prod_{i=1}^{\card\dpi}\(\sbpN_{i}\)^{b_i-1}  \prod_{i=1}^{\card{\tilde\pi}-1} (1-S^{(N)}_i); s^{(N)}_1>0,  \dots, s^{(N)}_{\card{\tilde\pi}}>0 \].
\end{align*}
Note that the r.h.s. above does not depend on $\pi$ and that (depending only on its block sizes) it 
is invariant under permutations of the partition $\tilde\pi$, i.e.
for any finite permutation $\perm$ we have
\begin{equation}\label{eq:limCoagRatePermInv}
\mat P^{(N,n)}({\pi,\coag(\pi, \tilde\pi)}) =\mat P^{(N,n)}({\pi,\coag(\pi, \sigma(\tilde\pi))}).
\end{equation}\par

On the other hand, let us assume for the moment that \eqref{eq:multLimitCoagProbs} exists for all $j\geq1$ and  $b_1\geq \dots b_j\geq1$. Then,  we have the recursion, 
\begin{align}\label{eq:limCoagRateRecursion}
&\phi(b_1,\dots,b_j,1) + \phi(b_1+1,b_2,\dots,b_j) + \cdots + \phi(b_1,\cdots,b_j+1)\nonumber\\
=& \lim_{N\to\infty} \frac{\E\[(S^{(N)}_j + (1-S^{(N)}_j))\(\prod_{i=1}^{j} \(\sbpN_i\)^{b_j-1}\)\prod_{i=1}^{j-1} \(1-S^{(N)}_i\); s^{(N)}_i>0, 1\leq i\leq j\] }{c_N}\nonumber\\
=&\phi(b_1,\cdots,b_j).
\end{align}\par

Equations \eqref{eq:limCoagRatePermInv} and \eqref{eq:limCoagRateRecursion} (for all $b_1\geq \dots b_j\geq1$) allow us to construct 
a measure $\mu$ on $\PS{\infty}$. For any $n\in\N$  and $\pi'\in\PS{\infty}$ let $\pi'\vert_n$ be the restriction of $\pi'$ to  $[n]$, and denote by
$\PS{\infty}(\pi)$ the set $\PS{\infty}(\pi)\coloneqq \{\pi'\in\PS{\infty}\colon \pi'\vert_n=\pi\}$. Then, setting 
for any partition $\pi\in\PS{n}$ with ordered block sizes $(b_1,\cdots,b_{\card \pi})$,
\begin{equation}\label{eq:coagMeaFiniteAddit}
 \mu\(\PS{\infty}(\pi)\)= \phi_{\card \pi}(b_1,\dots,b_{\card \pi})
\end{equation}
defines a measure on $\PS{\infty}$.

Indeed, the measure $\mu$ can be formally constructed via Caratheodory's extension theorem applied to the
measure $\hat\mu$ on the semi-ring 
$\mathscr{S}\coloneqq\{\PS{\infty}(\pi) : \pi\in\PS{n}\setminus\ze_n, \hspace{6pt} n\in\N\}$ defined analogously
through \eqref{eq:coagMeaFiniteAddit}. 
The finite-additivity of $\hat\mu$  follows easily from \eqref{eq:coagMeaFiniteAddit}. Infinite-additivity
follows from the fact that $\PS{\infty}$ is compact (see Lemma 2.6 in \cite{Bertoin2006}) 
and the fact that the elements of $\mathscr{S}$ are both open and closed. Indeed, the latter implies that if  $\mathscr{S}\ni A = \cup_{i=1}^\infty A_i$, $A_i \in\mathscr{S}$, then, in fact,
there exists a finite collection of indices $i_1,\dots,i_k$ such that $A=\cup_{j=1}^k A_{i_j}$. Thus, infinite-additivity follows from finite-additivity. The measure $\mu$ being
invariant under finite permutations (by construction),  Theorem 4.2 in \cite{Bertoin2006} ensures the existence
of a measure $\Xi$ on $\PS{[0,1]}$ such that 
$$\mu(\pi)=\mat Q^{(n)}_{\Xi}(\pi),\quad \forall n\in\N, \pi\in\PS{n}\setminus \ze_n.$$\par

Finally, it only remains to note that in fact it is enough to verify \eqref{eq:multLimitCoagProbs} for block sizes satisfying 
 $b_1\geq b_2\geq \dots\geq b_j\geq 2$. In other words, we need to show that in this case the limit in \eqref{eq:multLimitCoagProbs} will also exist for the greater
set of block sizes $b_1\geq b_2\geq \dots\geq b_j\geq 1$. The latter follows from the  recursion in \eqref{eq:limCoagRateRecursion} which
yields that if \eqref{eq:multLimitCoagProbs} exists for  $b_1\geq b_2\geq \dots\geq b_j\geq 2$ then it also must exist for
 $b_1\geq b_2\geq \dots\geq b_j\geq 2, b_{j+1}=1$, which through a second call to \eqref{eq:limCoagRateRecursion}, entails that \eqref{eq:multLimitCoagProbs}   must also exist for
$b_1\geq b_2\geq \dots\geq b_j\geq 2, b_{j+1}=1, b_{j+2}=1$, and so forth. 
\end{proof}
\subsection{From AC to AWF models} \label{sec:generalOffspring}

{The main result of this section says
that, in most cases, the sampling (of the parents) with and without replacement 
become equivalent in the limit when $N\to\infty$; so that the genealogies of asymmetric neutral populations can be mostly
studied by the results described so far for AWF populations.}
In the same spirit as in \cite{Schweinsberg2003},
let us consider an extension of our current setting. We assume that the total offspring size $\Sigma_N=\sum_{i=1}^N \osi{i}$ satisfies
\begin{equation}\label{sigmaN}
\Sigma_N\geq N
\end{equation}
almost surely and, to form the next generation, a sample without replacement of size $N$ among the newly produced children is taken. Our next proposition is inspired by and generalizes the relationship between the models of \cite{Schweinsberg2003}
and \cite{HuilletMohle2021}. 

\begin{proposition}\label{prop:equivWWRep}
Let $\(\osi{1},\dots,\osi{N}\)$ be a random vector of non-negative integers such that \eqref{sigmaN} holds.
Let $\tilde{\mat P}^{(N,n)}$ be the corresponding coagulation probabilities for a sample
of $n$ individuals: for any $\tilde\pi\in\PS{n}$,
\begin{equation*}
 \tilde{\mat P}^{(N,n)}(\tilde\pi) =  
\E\[\sum_{\substack{i_1,\dots,i_{\card{\tilde\pi}} \\ \text{all distinct}}}
 \frac{\(\osi{i_1}\)_{b_1} \cdots \(\osi{i_{\card{\tilde\pi}}}\)_{b_{\card{\tilde\pi}}}}
      {(\Sigma_N)_{n}}
\].
\end{equation*}
Also let
\begin{equation}\label{eq:cntilde}
\tilde c_N\coloneqq \tilde{\mat P}^{(N,2)}(\{\{1,2\}\})\equiv \E\[\sum_{i=1}^N
 \frac{\osi{i}\(\osi{i}-1\) }
        {\Sigma_N(\Sigma_N-1)}
\].
\end{equation}
Set $\rpvN_i={\osi{i}}/{\Sigma_N}$ and consider the corresponding AWF population with coagulation
probabilities ${\mat P}^{(N,n)} $ as in \eqref{PNpaintbox} and $c_N$ as in \eqref{eq:multcN}. Then for every $n\geq 2$ there exists $C(n)>0$ such that, for sufficiently large $N$
 { and in the usual operator norm}
\begin{equation}\label{eq:probDiffWWRep}
\norm{\tilde{\mat P}^{(N,n)} - {\mat P}^{(N,n)} } \leq C(n)\E\[\frac{1}{\Sigma_N}\].
\end{equation}
In particular, whenever either $c_N/\E\[{1}/{\Sigma_N}\]\to\infty$ or 
 $\tilde c_N/\E\[{1}/{\Sigma_N}\] \to\infty$ as $N\to\infty$, we have $c_N\sim\tilde c_N$ and, furthermore,  
both models have the same limit genealogy in distribution. 
\end{proposition}

\begin{remark}
Many examples in the literature, in particular those whose limit genealogies are not described by Kingman's coalescent,  
satisfy either $\Sigma_N/N\to \infty$ or $c_NN\to\infty$, both of which imply the condition $c_N/\E\[{1}/{\Sigma_N}\]\to\infty$. However, Proposition \ref{prop:equivWWRep} cannot be used to show the well known equivalence between some classical models in which $c_N\sim 1/N$. For example between
 the (symmetric) Wright-Fisher model and the corresponding Cannings model where $\osN$ has multinomial distribution of parameter $(p_i=1/N)_{1\leq i\leq N}$.
 Both of these models converge to Kingman's coalescent. 
\end{remark}

\begin{proof}
Stirling's approximation in the form $\Gamma(m+a)/\Gamma(m+b)=m^{a-b}(1+\bO{1/m})$ as $m\to\infty$, 
and the condition \eqref{sigmaN} yield, for some constant $C\equiv C(n)$
and sufficiently large $N$, the almost sure inequalities
$$
\Sigma_N^n\(1-C/\Sigma_N\) {\leq} (\Sigma_N)_n {\leq}\Sigma_N^n\(1+C/\Sigma_N\)
$$
so that
$$
 \tilde{\mat P}^{(N,n)} (\tilde\pi) = \E\[\sum_{\substack{i_1,\dots,i_{\card{\tilde\pi}} \\ \text{all distinct}}}
 \frac{\(\osi{i_1}\)_{b_1} \cdots \(\osi{i_{\card{\tilde\pi}}}\)_{b_{\card{\tilde\pi}}}}
      {\Sigma_N^{n}}
\] + \bO{\E\[\frac{1}{\Sigma_N}\]}.
$$\par
Using that for $x\in[0,1]$, $(1-x)^b-1{\geq -bx}$ we obtain, for $a>b-1$,
$$
(a)_b\geq (a-b+1)^b {\geq a^b - a^{b-1}b(b-1)}.
$$
Then for any $1\leq j\leq\ell \leq n$
and $n\geq b_1\geq \cdots \geq b_j\geq 2, b_{j+1}=\cdots=b_{\ell}=1$ we have, writing 
$D_{i,b}\coloneqq\(\(\rpvN_{i}\)^{b} - \frac{\(\rpvN_{i}\)^{b-1}}{{\Sigma_N}}b(b{-}1)\)\Ind{\osi{i}>b-1}$
and setting empty products to 1,  
\begin{align}\label{eq:equivRepl0}
&\sum_{\substack{i_1,\dots,i_{\ell} \\ \text{all distinct}}}
 \frac{\(\osi{i_1}\)_{b_1}}{\Sigma_N^{b_1}} \cdots \frac{\(\osi{i_{j}}\)_{b_{j}}}{\Sigma_N^{b_j}}
      \prod_{k=j+1}^\ell \rpvN_{i_k}
\\
&\geq \sum_{\substack{i_1,\dots,i_{\ell} \\ \text{all distinct}}}
      \prod_{k=1}^j D_{i_k,b_k}\prod_{k=j+1}^\ell \rpvN_{i_k}\Ind{\osi{i_k}\geq 1}
\nonumber\\
&=
\sum_{\substack{i_1,\dots,i_{\ell} \\ \text{all distinct}}}
\(\rpvN_{i_1}\)^{b_1} \Ind{\osi{i_1}>b_1-1}
\prod_{k=2}^j D_{i_k,b_k} 
\prod_{k=j+1}^\ell \rpvN_{i_k}\Ind{\osi{i_k}\geq 1}\nonumber \\ 
&\quad - b_1(b_1-1)
\sum_{i_1=1}^N \frac{\(\rpvN_{i_1}\)^{b_1-1}}{\Sigma_N}\Ind{\osi{i_1}>b_1-1}\sum_{\substack{i_2,\dots,i_{\ell}\subset[N]\setminus \{i_1\} \\ \text{all distinct}}}
\prod_{k=2}^j D_{i_k,b_k} 
\prod_{k=j+1}^\ell \rpvN_{i_k}\Ind{\osi{i_k}\geq 1}
.\nonumber
\end{align} 
Since $\sum_{i=1}^N \rpvN_i = 1$, we obtain the following bound for the second term above
\begin{align}\label{eq:equivRepl0.1}
&
\sum_{i_1=1}^N \frac{\(\rpvN_{i_1}\)^{b_1-1}}{\Sigma_N}\Ind{\osi{i_1}>b_1-1}\sum_{\substack{i_2,\dots,i_{\ell}\subset[N]\setminus \{i_1\} \\ \text{all distinct}}}
\prod_{k=2}^j D_{i_k,b_k} 
\prod_{k=j+1}^\ell \rpvN_{i_k}\Ind{\osi{i_k}\geq 1}\nonumber\\ 
&\leq \frac{1}{\Sigma_N} \sum_{\substack{i_1,\dots,i_{\ell} \\ \text{all distinct}}} 
      \rpvN_{i_1}\cdots\rpvN_{i_\ell}\leq \frac{1}{\Sigma_N}.
\end{align}
Iterating over $\{1,\dots,j\}$  and plugging the corresponding bounds \eqref{eq:equivRepl0.1} into \eqref{eq:equivRepl0} we obtain, for some $C(b_1,\dots,b_j)\equiv C >0$,
\begin{align}\label{eq:equivRepl1}
&\E\[\sum_{\substack{i_1,\dots,i_{\ell}\\ \text{all distinct}}}
 \frac{\(\osi{i_1}\)_{b_1}}{\Sigma_N^{b_1}} \cdots \frac{\(\osi{i_{j}}\)_{b_{j}}}{\Sigma_N^{b_j}}
      \prod_{k=j+1}^\ell \rpvN_{i_k}
\]\nonumber\\
&\geq 
\E\[\sum_{\substack{i_1,\dots,i_{\ell} \\ \text{all distinct}}}
\prod_{k=1}^j\(\rpvN_{i_k}\)^{b_k}  \Ind{ \osi{i_k}>b_k-1}\prod_{k=j+1}^\ell \rpvN_{i_k}
\] - \E\[\frac{C}{\Sigma_N}\].
\end{align}
Now to ``remove'' the indicator functions appearing inside the sum above, observe that 
\begin{align}\label{eq:equivRepl2}
&\E\[
\sum_{\substack{i_1,\dots,i_{\ell} \\ \text{all distinct}}}
\(\rpvN_{i_1}\)^{b_1} \cdots \(\rpvN_{i_\ell}\)^{b_\ell} \Ind{\cup_{k=1}^j \{\osi{i_k}\leq b_k-1\}}
\]\\&
\leq\E\[ \sum_{\substack{i_1,\dots,i_{\ell} \\ \text{all distinct}}}
\(\rpvN_{i_1}\)^{b_1} \cdots \(\rpvN_{i_\ell}\)^{b_\ell} \sum_{k=1}^j \Ind{\osi{i_k}\leq b_k-1}
\]\nonumber\\
&\leq \sum_{k=1}^j
      \E\[\sum_{\substack{i_1,\dots,i_{\ell} \\ \text{all distinct}}}
        \(\rpvN_{i_1}\)^{b_1} \cdots \(\rpvN_{i_{k-1}}\)^{b_{k-1}} \(\frac{b_k-1}{\Sigma_N}\)^{b_k} \(\rpvN_{i_{k+1}}\)^{b_{k+1}} \cdots \(\rpvN_{i_j}\)^{b_{j}} 
\]\nonumber \\
&\leq 
\sum_{k=1}^j \E\[\(\frac{b_k-1}{\Sigma_N}\)^{b_k} 
      \sum_{\substack{i_1,\dots,i_{\ell-1} \\ \text{all distinct}}}
        \(\rpvN_{i_1}\)^{b_1} \cdots \(\rpvN_{i_{\ell-1}}\)^{b_{\ell-1}} 
\] \leq \sum_{k=1}^j \E\[\(\frac{b_k-1}{\Sigma_N}\)^{b_k}\].\nonumber
\end{align}
Since $n\geq b_1\geq b_\ell\geq 1$ and $j\leq \ell\leq n$, the last expression becomes bounded by $n(n-1)^{n}\E\[\Sigma_N^{-1}\]$. Thus,
adding and subtracting \eqref{eq:equivRepl2} into \eqref{eq:equivRepl1}, we obtain for some $C(b_1,\dots,b_\ell)\equiv C'>0$,
\begin{align}\label{eq:equivRepl3}
&\E\[\sum_{\substack{i_1,\dots,i_{\ell} \\ \text{all distinct}}}
 \frac{\(\osi{i_1}\)_{b_1}}{\Sigma_N^{b_1}} \cdots \frac{\(\osi{i_{j}}\)_{b_{j}}}{\Sigma_N^{b_j}}
      \prod_{k=j+1}^\ell \rpvN_{i_k}
\]\nonumber\\
&\geq
\E\[\sum_{\substack{i_1,\dots,i_{\ell} \\ \text{all distinct}}}
\(\rpvN_{i_1}\)^{b_1} \cdots \(\rpvN_{i_\ell}\)^{b_\ell} 
\] - \E\[\frac{C'}{\Sigma_N}\].
\end{align}
Since we also have the trivial bound
$
(a)_b \leq  a^b,
$
equation \eqref{eq:equivRepl3} already yields \eqref{eq:probDiffWWRep}.\par
The last statement of the proposition follows from observing that if $c_N/\E\[\Sigma_N^{-1}\]\to\infty$ as $N\to\infty$
then $\E\[\Sigma_N^{-1}\]=\lo{c_N}$ and, furthermore,
equation \eqref{eq:probDiffWWRep} (with $n=2$) yields $c_N\sim\tilde c_N$. Thus
$$
\norm{\tilde{\mat P}^{(N,n)} - {\mat P}^{(N,n)}}=\lo{c_N}=\lo{\tilde c_N}.
$$
\end{proof}
Let us end this section with a simple application of the above theorem.
\begin{corollary}\label{cor:multiGenealogyDiscreteLimit}
In the setting of Proposition \ref{prop:equivWWRep}, let $\bm \rho^{(N)}$ be the vector $\bm \eta^{(N)}$ rearranged in decreasing order. Assume that
$\bm \rho^{(N)}$ converges in distribution in $\PS{[0,1]}$ to $\bm \rho^{(\infty)}$ as $N\to\infty$. Then 
the genealogies of both the AWF and the corresponding AC models converge in distribution
in the product topology for $\(\PS{n}\)^\N$ to a Markov chain with initial state $\ze_n$ and transition matrix given
by 
$$
\mat{P}^{(\infty,n)}(\pi, \coag(\pi, \dpi)) = \P\(\tilde\Pi^{(\infty,n)} = \dpi\),
$$
where $\tilde\Pi^{(\infty,n)}\in\PS{n}$ is obtained through the paintbox construction from $\bm\rho^{(\infty)}$.
\end{corollary}
\begin{proof}
Since $c_N\to \E[\sum_{i=1}^\infty (\rho^{(\infty)}_i)^2]>0$ and $\Sigma_N=N$, by Proposition \ref{prop:equivWWRep} we need only proof the result for the AWF model. 
Let $\tilde\Pi^{(N,n)}\in\PS{n}$ be a paint-box constructed from $\bm \rho^{(N)}$. Proposition 2.9 in \cite{Bertoin2006} readily gives,
for any $\dpi\in\PS{n}$,
$$\P\(\tilde\Pi^{(N,n)} = \dpi\) = \P\(\tilde\Pi^{(\infty,n)} = \dpi\) + \lo{1},$$ 
 which yields the condition
of Theorem \ref{th:multWeakConv} for the case $\lim_{N\to\infty}c_N>0$ with $c=\E[\sum_{i=1}^\infty (\rho^{(\infty)}_i)^2]$.
\end{proof}

\section{Simpler criteria}\label{seq:largeRepEvents}
In this section we provide a simpler sufficient condition for \eqref{eq:TaylorMassPartition} to hold
for AWF models whose genealogies converge to a $\Lambda$-coalescent (\cite{Pitman99, Sagitov99}). A $\Lambda$-coalescent corresponds to the case in which the coagulation measure $\Xi$ is supported on mass partitions of the form $(p,0,0,\cdots)$,
and we denote by $\Lambda$ the pushforward of $\Xi$ after projecting into the first coordinate $p\in(0,1)$. 
In particular, the partition-valued process evolves through simple coaugulations, in which a single group of $b$ blocks merge together at rate 
$$
\lambda_{n,b}\coloneqq \int_0^1 p^{b-2}(1-p)^{n-b}  \Lambda(dp)
$$
whenever there are $n$ blocks in the system.\par

There are two simple heuristics from distinct population dynamics
which motivate this section (and the appearance of $\Lambda$-coalescents in general): 1) all the individuals in the population have the same
fitness but occasionally, by mere chance, there occur  large reproductive events where a single 
individual takes up a large fraction of the offspring in the next generation, and
2) the offspring distribution is highly asymmetric 
and the family frequencies given by the mass-partition $\rpv$ 
are dominated by the largest family frequency (say $\rho^{(N)}_1$) associated to the individual with the highest reproductive success. 
In both of these scenarios 
we expect that all the coagulations that remain observable in the limit correspond to very large reproduction 
events of single individuals along the generations (see \cite{Schweinsberg2003, HuilletMohle2021, CortinesMallein2017}); 
these coagulations must then be driven by $\rho^{(N)}_1$. 
In practice we would rather allow more flexibility 
and prove that all the coagulations that ``survive'' in the limit are driven by the frequency
$\rpvN_{1}$. This entails that $\rpvN_{1}$ should approximate $\rho^{(N)}_1$ as $N\to\infty$, but working
with $\rpvN_{1}$ instead of $\rho^{(N)}_1$ may ease the study of the underlying genealogy of the population 
(also note that, the distribution of $\rpv$ having no restrictions, our setup does include the case
where $\rpvN_1=\rho^{(N)}_1$). Of particular interest is the case where $\(\rpvN_{1},\dots,\rpvN_{N}\)$ is a
(normalized) size-biased pick from a random mass partition for which the distribution may be explicitly known (see Section \ref{sec:multPDcase} for a more general example of this). \par
The first result, an easy consequence of Theorem \ref{th:multWeakConv}, reads as follows.
\begin{proposition}\label{prop:multSPlambda}
Assume that there exists $i\in[N]$ such that, for every $b\geq  2$,
\begin{equation}\label{critsimple}
\E\[\(\rpvN_i\)^b\] \sim c_N \frac{ \lambda_{b,b}}{\lambda_{2,2}}.
\end{equation}
Then the genealogy $\stp{\cpNn_{\lfloor t/c_N\rfloor}}$ of the corresponding AWF population
converges weakly in the Skorohod space $D([0,\infty), \PS{n})$ to a coalescent with characteristic measure $\Lambda$.
\end{proposition}
\begin{proof}
With no loss of generality, suppose that $i=1$.
The definition of $c_N$ in \eqref{eq:multcN},
plus the hypothesis \eqref{critsimple} when $b=2$ give
\begin{align}\label{ocN}
\E\[\sum_{i=2}^N\(\rpvN_i\)^2\] = \lo{c_N}.
\end{align}
Also note that, for the sum appearing in \eqref{PNpaintbox} but restricted to  indices greater than 2, we have
\begin{align}\label{eq:multiPNrecursiveBound2}
&\sum_{\substack{i_1,i_2,\dots,i_j \geq 2\\ \text{all distinct}}} \(\rpvN_{i_1}\)^{b_1}\cdots\(\rpvN_{i_j}\)^{b_j}\nonumber
\\= &\sum_{\substack{i_1,i_2,\dots,i_{j-1}\geq 2 \\ \text{all distinct}}} 
 \(\rpvN_{i_1}\)^{b_1}\cdots\(\rpvN_{i_{j-1}}\)^{b_{j-1}} \sum_{i_j\not\in \{i_1,\dots,i_{j-1}\}} \(\rpvN_{i_j}\)^{b_j}\nonumber
\\\leq& \sum_{\substack{i_1,i_2\dots,i_{j-1}\geq 2 \\ \text{all distinct}}}
\(\rpvN_{i_1}\)^{b_1}\cdots\(\rpvN_{i_{j-1}}\)^{b_{j-1}} \nonumber
\\\vdots&\nonumber
\\\leq &\sum_{i=2}^N \(\rpvN_{i}\)^{b_1}\leq \sum_{i=2}^N \(\rpvN_{i}\)^{2}.
\end{align}

The latter implies, for any $\pi\in\PS{n}$ and coagulation increment $\dpi\in\PS{\card{\pi}}\setminus \ze_{\card{\pi}}$ 
with ordered block sizes $b_1\geq \dots \geq b_{j}$, $j=\card{\dpi}$, 
\begin{align*}
{\mat P}^{(N,n)}({\pi,\coag\(\pi,\dpi\)})
&= \E\[\sum_{\substack{i_1=1,i_2,\dots,i_{j}\\\text{all distinct}}} \(\rpvN_{i_1}\)^{b_1}\cdots\(\rpvN_{i_j}\)^{b_j}\] + \lo{c_N}\\
&=\E\[\(\rpvN_{1}\)^{b_1} \sum_{\substack{i_2,\dots,i_j\geq2 \\ \text{all distinct}}} \(\rpvN_{i_2}\)^{b_2}\cdots\(\rpvN_{i_j}\)^{b_j}\]+\lo{c_N}.
\end{align*}

Again, using \eqref{eq:multiPNrecursiveBound2} and \eqref{ocN}, we have
\begin{align*}
&\E\[\(\rpvN_{1}\)^{b_1}\sum_{\substack{i_2,\dots,i_j\geq2\\\text{all distinct}}} \(\rpvN_{i_2}\)^{b_2}\cdots\(\rpvN_{i_j}\)^{b_j}\]\\
&\quad\quad = \begin{cases}
\lo{c_N} &\text{ if }b_2\geq 2,\\
\E\[\(\rpvN_{1}\)^{b_1}(1-\rpvN_{1})^{\card{\pi}-b_1}\] &\text{ otherwise}.
\end{cases}
\end{align*}

It thus remains to prove that 
\begin{equation}\label{eq:LambdaConvRates}
\E\[\(\rpvN_{1}\)^{b_1}(1-\rpvN_{1})^{\card{\pi}-b_1}\]\sim c_N\frac{\lambda_{\card{\pi},b_1}}{\lambda_{2,2}}
\end{equation}
 whenever $\card{\pi}\geq b_1\geq2$.
 This is true when $\card{\pi}= b_1$, thanks to the hypothesis \eqref{critsimple}.
 This can be shown in whole generality thanks to an induction over the values of $\card{\pi}- b_1\geq0$.
Given the well-known recursion formula $\lambda_{\card{\pi},b_1}=\lambda_{\card{\pi}+1,b_1}+\lambda_{\card{\pi}+1,b_1+1}$  (Lemma 18 \cite{Pitman99}) on the one side, 
and the recursion 
\begin{align*}
\E\[\(\rpvN_{1}\)^{b_1}(1-\rpvN_{1})^{\card{\pi}-b_1}\] =&
\E\[\(\rpvN_{1}\)^{b_1}(1-\rpvN_{1})^{\card{\pi}+1-b_1}\]\\+
&\E\[\(\rpvN_{1}\)^{b_1+1}(1-\rpvN_{1})^{\card{\pi}+1-b_1-1}\],
\end{align*}
on the other, it is not difficult to see that \eqref{eq:LambdaConvRates} for any choice of $\card{\pi}- b_1=k\geq0$ implies the same equivalence for the
cases $\card{\pi}-b_1=k+1$. 
\end{proof}

The above proposition provides a mild improvement to Lemma 3.1 in \cite{CortinesMallein2017} since, in our case, the family frequencies need not to be arranged in decreasing order. We now present our simplified criterion for the convergence to Kingman's coalescent which is in correspondence
with their Lemma 3.2. Note that, in this case, and as is common in scenarios with convergence to Kingman's coalescent, it is no longer the case that the coagulation events that remain in the limit are driven by a single family frequency $\eta^{(N)}_i$.

\begin{proposition}\label{prop:multiKingman2}
Assume that
$$
\E\[\(\rpvN_2\)^3 + \dots + \(\rpvN_N\)^3\] = \lo{c_N},
$$
and that there exists $\beta>2$ such that
$$
   \E\[\(\rpvN_1\)^\beta\] = \lo{c_N}.
$$ 
Then
\begin{equation}\label{eq:kingCondition}
\E\[\(\rpvN_1\)^3 + \dots + \(\rpvN_N\)^3\] = \lo{c_N}.
\end{equation}
If, in addition, $c_N\to0$, then the genealogy $\stp{\cpNn_{\lfloor t/c_N\rfloor}}$ of the corresponding AWF population
converges weakly in the Skorohod topology for
$D([0,\infty), \PS{n})$ to Kingman's coalescent.
\end{proposition}
\begin{proof}
We first prove that the second hypothesis implies $$\E\[\(\rpvN_1\)^3\]=\lo{c_N}.$$ 
We follow an argument of \cite{CortinesMallein2017}. The result
is trivial if $\beta\leq 3$, we thus assume $\beta>3$.
Observe that for $\lambda\in(0,2)$ we have, using Hölder's inequality,
\begin{align*}
\E\[\(\rpvN_{1}\)^{3}\]
&= \E\[\(\rpvN_{1}\)^{\lambda}\(\rpvN_{1}\)^{3-\lambda}\]\\
&\leq \E\[\(\rpvN_{1}\)^2\]^{\lambda/2}\E\[\(\rpvN_{1}\)^{\frac{2(3-\lambda)}{2-\lambda}}\]^{1-\lambda/2}
\end{align*}
so that, choosing $\lambda\in(0,2)$ to ensure $\frac{2(3-\lambda)}{2-\lambda}>\beta$,  we obtain by the second hypothesis, 
\begin{align*}
\frac{\E\[\(\rpvN_{1}\)^{3}\]}{c_N}
&\leq \frac{\E\[\(\rpvN_{1}\)^{\frac{2(3-\lambda)}{2-\lambda}}\]^{1-\lambda/2}}{c_N^{1-\lambda/2}}\\
&\leq \(c^{-1}_N\E\[\(\rpvN_{1}\)^{\beta}\]\)^{1-\lambda/2},
\end{align*}
converging to 0 when ${N\to\infty}$. 
Combining with the first hypothesis, we conclude \eqref{eq:kingCondition}.
\par
Now, let  $b_1\geq\cdots b_j\geq 1$ be the ordered block sizes of the coagulation increment $\pi'$ in \eqref{PNpaintbox}.  By a similar computation as that in \eqref{eq:multiPNrecursiveBound2} we have
$$
\sum_{\substack{i_1,i_2,\dots,i_j \nonumber\\ \text{all distinct}}} \(\rpvN_{i_1}\)^{b_1}\cdots\(\rpvN_{i_j}\)^{b_j}
\leq\sum_{i=1}^N \(\rpvN_{i}\)^{b_1}.
$$
This, together with \eqref{PNpaintbox} and \eqref{eq:kingCondition} yield
$
\PN_{\pi,\coag\(\pi,\dpi\)} = \lo{c_N}
$
whenever $b_1\geq 3$.
Similarly, whenever $b_2\geq 2$ and using Jensen's inequality, we have
\begin{align*}
&\sum_{\substack{i_1,i_2,\dots,i_j \nonumber\\ \text{all distinct}}} \(\rpvN_{i_1}\)^{b_1}\cdots\(\rpvN_{i_j}\)^{b_j} \leq \sum_{\substack{i_1,i_2\\\text{distinct}}} \(\rpvN_{i_1}\)^{2}\(\rpvN_{i_2}\)^{2}\\
&\leq \( \sum_{i=1}^N (\rpvN_i)^2\)^2 \leq \sum_{i=1}^N \rpvN_i (\rpvN_i)^2.
\end{align*}
Thus, \eqref{eq:kingCondition} now implies $\PN_{\pi,\coag\(\pi,\dpi\)} = \lo{c_N}$ whenever $b_2\geq 2$. Since, by hypothesis, we also have $c_N\to0$, the convergence to Kingman's coalescent follows by ii) of Theorem \ref{th:multWeakConv}.
\end{proof}

\section{Some examples for neutral evolution}\label{neutralmodels}

\subsection{The Eldon and Wakeley model}
 Let us consider a simple example of a non-multinomial asymmetric Cannings model, a generalization of the celebrated model
studied in \cite{EldonWakeley2006}, whose genealogy  can be easily studied using our criteria. \par
Let $\(Y_t^{(N)}\)_{t\in\N}$ be a collection of i.i.d. r.vs. on $[0,1]$. 
Consider a population where, at any given generation $t$, a randomly
chosen individual produces $\floor{Y_t^{(N)}N}$ copies of itself which then replace 
$\floor{Y_t^{(N)}N}$ randomly chosen individuals 
 (uniformly and without replacement) in the next generation. Formally, the corresponding random vector of offspring sizes can
be constructed as follows. We set $\osi{K}=\floor{Y^{(N)}N}$ for some random variable $Y^{(N)}$ in $[0,1]$ and some uniformly-distributed 
index $K\in[N]$. Conditional on those random variables,
we set $\osi{i}=1$ for all $i\in I$, where $I$ is a uniform random sample of size $N-\floor{Y^{(N)}N}$  
from $[N]\setminus \{K\}$ without replacement. Finally we set
$\osi{i}=0$ for all $i\in [N]\setminus \{\{K\}\cup I\}$. 
\begin{corollary}
Assume that the distribution of $Y^{(N)}$ is given by 
$$\P(Y^{(N)}\in dy)=\frac{y^{-2}\Lambda(dy)}{\int_{N^{\frac{\epsilon-1}{2}}}^1 y^{-2}\Lambda(dy)}\Ind{y>N^{\frac{\epsilon-1}{2}}}$$
for some finite measure $\Lambda$ on $[0,1]$ and some $\epsilon\in(0,1)$.
Then $ c_N\sim \E\[\(Y^{(N)}\)^2\]$ and $\stp{\cpNn_{\floor{t/ c_N}}}$ 
converges weakly in $D([0,\infty), \PS{n})$ to the $\Lambda$-coalescent. 
\end{corollary}
\begin{proof}
By the non-heritability condition, the model is equivalent, in terms of its genealogy and in distribution,
to the case in which $K{=}1$ almost surely. Furthermore, recalling its definition in \eqref{eq:cN},
\begin{align*}
 c_N&=  
\frac{1}{N(N-1)}\E[{\floor{Y^{(N)}N}(\floor{Y^{(N)}N}-1)}] \\
&\sim \E\[(Y^{(N)})^2\]
=\frac{\Lambda([N^{\frac{\epsilon-1}{2}},1])}{\int_{N^{\frac{\epsilon-1}{2}}}^1 y^{-2}\Lambda(dy)}.
\end{align*}
Since 
$\int_{N^{\frac{\epsilon-1}{2}}}^1 y^{-2}\Lambda(dy)\leq N^{1-\epsilon}\Lambda([0,1])$ we have
$N c_N\to\infty$. Then, by Proposition \ref{prop:equivWWRep} we may work with the AWF version of the model with 
$$\eta_1^{(N)}=\frac{\floor{Y^{(N)}N}}{N}= Y^{(N)}+\bO{N^{-1}}=Y^{(N)}(1+\bO{N^{-1+\frac{1-\eps}{2}}}).$$ 
Then, by Proposition \ref{prop:multSPlambda},
we need only observe that, for all $b\geq2$,
$$
\lim_{N\to\infty} \frac{\E\[(Y^{(N)})^b\]}{\E\[(Y^{(N)})^2\]}
= \lim_{N\to\infty} \frac{\int_{N^{\frac{\epsilon-1}{2}}}^1 y^{b-2}\Lambda(dy)}{\int_{N^{\frac{\epsilon-1}{2}}}^1 y^{2-2}\Lambda(dy)}
=\frac{\lambda_{b,b}}{ \lambda_{2,2}}.
$$
\end{proof}
\subsection{Evolution with recurrent bottlenecks}\label{sec:BN}

In section \ref{sec:generalOffspring} we considered a model where the population was left free to increase over $N$ and get back to the equilibrium immediately. In this section we model the inverse phenomenon, now the population suddenly and drastically decreases and gets back to the equilibrium immediately after this catastrophic event.
Here we adapt  our setting to a generalization of the Wright-Fisher model with recurrent bottlenecks introduced in \cite{CasanovaMiroJegousse2020}. 
The general picture is the following. At every discrete time, the population undergoes a drastic decrease of its size for only one generation and with a small probability. 
The next generation then descends from a finite number of individuals that, as $N\to\infty$, is asymptotically independent of $N$.


More precisely, we consider an AWF model with  family frequencies vector
$$\bm\eta^{(N)}=\bm{\bar\eta}^{(N)}B^{(N)}+\bm{\hat\eta}^{(N)}(1-B^{(N)})$$
where $B^{(N)}$ is a Bernoulli variable that determines whether a bottleneck event occurs.
When a bottleneck does occur, the next generation will descend from a random number of individuals $Y^{(N)}<N$ that is independent of $B^{(N)}$, meaning that $\bar\eta_i^{(N)}=0$ if $i>Y^{(N)}$ (or $\sum_{i=1}^{Y^{(N)}}\bar\eta_i^{(N)}=1$).
We denote by $\bar\nu_k$ the conditional law of the random probability vector $(\bar\eta_1^{(N)}, \dots, \bar\eta_k^{(N)})$ given $Y^{(N)}=k$ for $k\geq1$.
We assume that this conditional law does not depend on $N$. Also, we assume that
the probability that a bottleneck event occurs is given by
$$\P(B^{(N)}=1)=\frac{\sum_{k=1}^{b_N}F(k)}{a_n}$$
and that, when a bottleneck occurs,  there are $k\in[b_N]$ individuals left alive with probability  
$$\P(Y^{(N)}=k)=\frac{F(k)}{\sum_{k=1}^{b_N}F(k)},$$
 where $F$ is a (possibly infinite)  measure on $\N$. Here $(a_N)_{N\ge1}$ and $(b_N)_{N\ge1}$ are two diverging sequences  satisfying
 $\sum_{k=1}^{b_N}F(k)=o(a_n)$ and $b_N=o(N)$.
Intuitively, these two assumptions imply that bottlenecks happen with a small probability and that their intensity reduces significantly the population size. \par
Let us introduce the probability $\bar c_k$ that two randomly chosen individuals have the same parent just after a bottleneck of final size $k$ occurred, this is given by
$$\bar c_{k}
=\E\left[\sum_{i=1}^k\(\bar\eta_i^{(N)}\)^2\given{Y^{(N)}=k} \right].
$$
Also consider the same probability when no bottleneck occurred, that is
$$\hat c_{N}
=\E\left[\sum_{i=1}^N\(\hat\eta_i^{(N)}\)^2\right].
$$
As will be seen in the next result, two types of genealogies can appear in the limit depending on the frequency of the bottleneck events compared to the rescaling time of the AWF model driven only by $\bm{\hat\eta^{(N)}}$ (which would correspond to a scenario without bottleneck events).
In particular, when bottleneck events dominate the coalescence dynamics, the limit genealogies are described by a $\Xi$-coalescent such that, at rate $F(k)$, a merging event is performed according to a paintbox driven by the mass partition $\bar\nu_k$.
Denoting by $\varrho_\rho$ the law of a paintbox according to a random mass partition $\bm\rho$, in this case the coagulation matrix is given by
\begin{equation}\label{RFcoal}
\mat \bar Q^{(n)}(\dpi) \coloneqq\sum_{k=1}^\infty  F(k) \int_{\PS{[0,1]}} \varrho_\rho(\tilde\pi) \bar\nu_k(d\rho).
\end{equation}

\begin{corollary}
Assume that $0<\lim_{N\to\infty}\sum_{k=1}^{b_N} F(k)\bar c_{N,k}<\infty$  and that the AWF model driven by $\bm{\hat\eta^{(N)}}$ is in the domain of attraction of a coalescent with coagulation matrix $\mat Q^{(n)}_{\Xi}$, in the sense of Theorem \ref{th:multWeakConv}.
\\
i)
If $\hat c_Na_N\to0$, then,  as $N\to\infty$ the rescaled genealogy $\stp{\cpNn_{\floor{ta_N}}}$ converges weakly in the Skorohod space $D([0,\infty), \PS{n})$ towards
a coalescent with coagulation matrix $\mat \bar Q^{(n)}$ defined in \eqref{RFcoal} and initial state $\ze_n$.
\\
ii)
If $\hat c_Na_N\to\infty$, then, as $N\to\infty$ the rescaled genealogy $\stp{\cpNn_{\floor{t/\hat c_N}}}$ converges weakly in the Skorohod space $D([0,\infty), \PS{n})$ towards a coalescent with coagulation matrix $\mat Q^{(n)}_{\Xi}$ and initial state $\ze_n$.
\\iii) If $\hat c_Na_N\to \ell>0$, then, as $N\to\infty$ the rescaled genealogy $\stp{\cpNn_{\floor{ta_N}}}$ converges weakly in the Skorohod space $D([0,\infty), \PS{n})$ towards a coalescent with coagulation matrix 
$ \mat \bar Q^{(n)}_{\Xi}+
\ell\mat Q^{(n)}_{\Xi}$.
\end{corollary}

\begin{proof}
Let us first compute $c_N$ (in \eqref{eq:generalcN}) by conditioning on the value of $B^{(N)}$. We have
\begin{align*}
c_N&=\sum_{k=1}^{b_N}\bar c_{N,k}\P( Y^{(N)}=k, B^{(N)}=1)+\hat c_N\P(B^{(N)}=0)\\
&=\frac{\sum_{k=1}^{b_N}F(k)\bar c_{N,k}}{a_N}+\hat c_N(1-\frac{\sum_{k=1}^{b_N}F(k)}{a_N}).
\end{align*}

In case i), we get that
$$c_N\sim\frac{\sum_{k=1}^{b_N}F(k)\bar c_{N,k}}{a_N}$$
and since $0<\lim_{N\to\infty}\sum_{k=1}^{b_N} F(k)\bar c_{N,k}<\infty$, the limit coagulation matrix obtained consists in the paintbox representation obtained in \eqref{RFcoal}.

In case ii), we get that
$$c_N\sim\hat c_N$$
and we are coming back to the AWF model driven by $\bm\hat\eta^{(N)}$, with limit coagulation matrix $\mat Q^{(n)}_{\Xi}$.

Finally, in case iii), we get that
$$c_N\sim\frac{\sum_{k=1}^{b_N}F(k)\bar c_{N,k}}{a_N}+\frac\ell{a_N}$$
and the effects of both coagulation matrices add up in the limit.
\end{proof}

We provide two examples.
\begin{itemize}
\item Wright-Fisher model with Wright-Fisher bottlenecks: The case where $\bm\bar\eta^{(N)}=(\frac1{Y^{(N)}},\dots,\frac1{Y^{(N)}},0,\dots,0)$, $\bm\hat\eta^{(N)}=(\frac1N,\dots,\frac1N)$, and also $a_N=N^\alpha$ for some $\alpha\in(0,1)$ and $F$ is any measure on $\N$ such that $\sum F(k)/k<\infty$, is similar to the Wright-Fisher model with short drastic recurrent bottlenecks studied in \cite{CasanovaMiroJegousse2020}. The limit genealogies are described by the so-called symmetric coalescent. Its coagulation matrix is given by \eqref{RFcoal} where $\bar\nu_k=\delta_{(\frac1k,\dots,\frac1k)}$.
\item Wright-Fisher model with general Dirichlet bottlenecks:  A  normalized random partition with $K$ components is a partition of the form $(\frac {W_1}{S_K},\dots, \frac {W_K}{S_K})$ where $(W_1, W_2,\dots)$ is a family of i.i.d. random variables with finite variance and $S_{K}=\sum_{i=1}^KW_i$. The model in which $\bm\bar\eta^{(N)}$ is a normalized random partition with $Y^{(N)}$ components, $\bm\hat\eta^{(N)}=(\frac1N,\dots,\frac1N)$, and also $a_N=N^\alpha$ for some $\alpha\in(0,1)$ and $F$ is any measure on $\N$ such that $\sum F(k)/k<\infty$, provides an illustrative example where the limit genealogies are described by the Dirichlet coalescents introduced  in \cite{CasanovaMiroJegousseSchertzer2022}.
\end{itemize}

\section{An example for selective evolution}\label{sec:multPDcase}
In this section we study the case in which the family frequencies of the AWF model 
are constructed from a size-biased pick
$(\tilde V_1, \tilde V_2,\dots)$ of a Poisson-Dirichlet (PD) random mass partition of parameters 
$(\alpha,\theta)$, $0<\alpha<1, \theta>-\alpha$ \cite{PitmanYor97}, by setting, for $1\leq i \leq N$,
\begin{equation}\label{eq:PDdmpNdef}
\rpvN_i = \frac{\tilde V_{i}^\ppwr}{\sum_{j=1}^N \tilde V_j^\ppwr },
\end{equation}
where $\ppwr\in\mathbb{R}$ is a parameter of the model.\par
The case $\ppwr=\alpha$ of the following theorem was proved  in \cite{CortinesMallein2017}.
Here we extend their results by following their main heuristics  but using different technical
arguments such as  applying  Propositions \ref{prop:multSPlambda} and \ref{prop:multiKingman2}
instead of their Lemmas 3.1 and 3.2, thus avoiding having to deal with the largest family frequency among $\rpv$ directly.
To ease notations we write $$u_N=\sum_{i=1}^N i^{-\gamma/\alpha}$$
in the whole section.

\begin{theorem}\label{th:PDmultGenealogy}
Assume $\alpha/2<\ppwr\leq \alpha$ and set
$
L_N = \ell_{\alpha,\theta,\ppwr}u_N^{1+\frac{\theta}{\alpha}}$  where 
$$
\ell_{\alpha,\theta,\ppwr}^{-1}=
\frac{\alpha}{\ppwr}\frac{\Gamma(1-\alpha)^{\frac{\theta}{\alpha}}}{\Gamma(1+\ppwr-\alpha)^{1+\frac{\theta}{\alpha}}}
    \frac{\Gamma\(\frac{\alpha+\theta}{\alpha}(1-\frac{\ppwr}{\alpha})+1\)}
         {\Gamma\(\(\alpha+\theta\)\(1-\frac{\ppwr}{\alpha}\) +1\)}
   \frac{\Gamma(1+\theta) \Gamma\(1-\frac{\theta}{\alpha}\)}
        {\Betaf\(1-\frac{\theta}{\alpha},1+\frac{\theta}{\alpha}\)}.
$$
Let $\stp{\cpNn_{\floor{t/c_N}}}$ be the genealogy of the AWF model with family frequencies as in \eqref{eq:PDdmpNdef}.\\
i) If $\theta\in(-\alpha,\alpha)$, then $c_N\sim \(1-\frac{\theta}{\alpha}\)/L_N$ and $\stp{\cpNn_{\floor{t/c_N}}}$ 
converges weakly in the Skorohod topology for $D([0,\infty), \PS{n})$ to the Beta$\(1-\frac{\theta}{\alpha},1+\frac{\theta}{\alpha}\)$-coalescent, as $N\to\infty$.\\
ii) If $\theta=\alpha$, then $c_N=\bO{u_N^{-2}\log\(u_N\)}$, 
whereas if $\theta>\alpha$, then $c_N=\bO{u_N^{-2}}$. In both cases $\stp{\cpNn_{\floor{t/c_N}}}$ converges weakly to Kingman's coalescent, as $N\to\infty$.
\end{theorem}

Interestingly, the choice of $\ppwr$ does not affect the shape of the limit genealogy, but only the time rescaling. As explained in the following section, the above Theorem \ref{th:PDmultGenealogy} will provide a characterization of the limit genealogy of a population model that incorporates noisy natural selection (see the upcoming Theorem \ref{th:genealogiesExpModel}).

{\subsection{Motivation: exponential models}\label{sec:exponentialModel}
Here we briefly explain how Theorem \ref{th:PDmultGenealogy} describes the limit genealogy of a slightly more general
{exponential model (EM)} than those studied in \cite{CortinesMallein2017,CortinesMallein2018}, 
which is in fact a simple combination of them both. 
In order to avoid plain repetitions of their arguments we will heavily refer to their
results and proofs, and only provide the main heuristics therein.\par
We now describe the $(N,\beta,\kappa)$-EM, a generalization of the $(N,\beta)$-EM described in \cite{CortinesMallein2017}. 
Consider a population of particles 
positioned on the real line that evolve through discrete generations $t\in\N$.  
Every generation is of size $N$ and is constructed from
the previous generation through {branching} and  {selection} steps. The positions of the particles, say 
 $(X_1^{(N)}(t),\dots,X_N^{(N)}(t))$, give individuals' fitness levels: higher position
gives a greater probability of having more descendants in the next generation. 
Let $\kappa>0$ be a parameter of the model. During the branching step an individual at position $x$  will 
be replaced, independently from the other particles and from the previous generations, 
by a countable number of children whose positions are given by a PPP of intensity $e^{-(s-\kappa x)}\dif{s}$.
Let $\beta>1$ be another parameter of the model. During the selection step, $N$ of the newly produced particles 
are sampled without replacement, and with the probability of picking a child at position
$x$ being proportional to $e^{\beta x}$. Proposition 1.3 in \cite{CortinesMallein2017} ensures that the model is well defined 
(i.e. the selection probabilities
are well defined) precisely when $\beta>1$. The case $\kappa=1$ corresponds to the $(N,\beta)$-EM in which
children particles are centered around the position of the parent, as opposed to the general case in which children are
positioned around the scaled position of the parent $\kappa x$. On the other hand, if instead of the noisy selection step
just described we performed a ``perfect'' selection step, 
i.e. if we selected the $N$ rightmost children (which corresponds to setting $\beta=\infty$), then we recover the 
exactly solvable model studied in \cite{CortinesMallein2018}. Another variation of the exponential model is also studied by
 \cite{SchertzerWences2023}, in which the selection step is performed on the subset of the 
$\lfloor N^\chi \rfloor$, $\chi\geq1$, rightmost children produced during the branching step. In this case the strength
of selection $\beta$ may range in $[0,\infty)$, finding both discrete and
continuous time coalescent processes in the limit genealogy, depending on the choice of $\chi$ and $\beta$.\par

In the following proposition we recapitulate without formal proofs some of the results in 
\cite{CortinesMallein2017} which will establish a clear
connection between the genealogy of the $(N,\beta,\kappa)$-EM and the AWF model of Theorem \ref{th:PDmultGenealogy}.

\begin{proposition}\label{prop:EMasAWF}
In the $(N,\beta, \kappa)$-EM, 
a particle produced during the branching step of generation $t$
is a child of the particle at position $X_i^{(N)}(t)$ with probability, conditional on $(X_1^{(N)}(t),\dots,X_N^{(N)}(t))$ and independently from its own position and from the rest of its siblings, given
by
$
\frac{e^{\kappa X^{(N)}_i(t)}}{\sum_{j=1}^N e^{\kappa X^{(N)}_j(t)}}.
$
Furthermore, the sequence of random vectors 
$\(\frac{e^{\kappa X^{(N)}_1(t)}}{\sum_{j=1}^N e^{\kappa X^{(N)}_j(t)}}, \dots, \frac{e^{\kappa X^{(N)}_N(t)}}{\sum_{j=1}^N e^{\kappa X^{(N)}_j(t)}}\)_{t\in\N}$ is independent and identically distributed as 
$
\(\frac{\tilde V_{1}^{\kappa/\beta}}{\sum_{j=1}^N \tilde V_j^{\kappa/\beta} },\dots,\frac{\tilde V_{N}^{\kappa/\beta}}{\sum_{j=1}^N \tilde V_j^{\kappa/\beta} } \)
$
where $(\tilde V_1,\tilde V_2, \cdots)$ is a size-biased pick of a PD$(1/\beta,0)$ random variable. 
\end{proposition}
\begin{proof}[Sketch of the proof ]
This follows from Lemma 1.6 and Proposition 4.1 in \cite{CortinesMallein2017}; see also the proof of Theorem 1.2 therein. 
Here we only present the
main heuristics. 
To ease the notation, let us set $\kappa=1$.

The first assertion follows from well-known identities of PPPs. Indeed, 
by superimposing $N$ independent PPPs of respective intensities 
$e^{-(s-X^{(N)}_i(t))}ds, 1\leq i\leq N,$ one can produce a single (joint) PPP of intensity
$e^{-(s-X^{(N)}_{eq}(t))}ds$ where $X^{(N)}_{eq}(t)=\log\(\sum_{i=1}^N e^{X^{(N)}_i(t)}\)$. Then, a 
particle at position $x$ of the joint PPP happens to be a particle of the $i$th PPP,
i.e. is a daughter of the particle at position $X^{(N)}_i(t)$,
with probability 
$$
\frac{e^{-(x-X^{(N)}_i(t))}}{\sum_{j=1}^N e^{-(x-X^{(N)}_j(t))}}= \frac{e^{X^{(N)}_i(t)}}{\sum_{j=1}^N e^{X^{(N)}_j(t)}}.
$$

The i.i.d. nature of the sequence $\(\frac{e^{X^{(N)}_1(t)}}{\sum_{j=1}^N e^{X^{(N)}_j(t)}}, \dots, \frac{e^{X^{(N)}_N(t)}}{\sum_{j=1}^N e^{X^{(N)}_j(t)}}\)_{t\in\N}$
follows from observing that the positions of the children produced during the branching step of generation $t$ are distributed like 
$X^{(N)}_{eq}(t)+z_i, i\ge1$ where $(z_1,z_2,\dots)$ are the atoms of a PPP of intensity $e^{-s}ds$, arranged in decreasing order, and this point process
is independent of all the previous generations.  
Hence, we can write the sampling probabilities of the respective selection step (to produce generation $t+1$) as
$$
\( \frac{e^{\beta X^{(N)}_{eq}(t) + \beta z_i}}{\sum_{{j\ge1}} e^{\beta X^{(N)}_{eq}(t) + \beta z_{j}}}  \)_{i\in\N}
=\( \frac{e^{\beta z_i}}{\sum_{{j\ge1}} e^{ \beta z_{j}}}  \)_{i\in\N}.
$$
Observe that the latter does not depend on $(X^{(N)}_1(t),\dots,X^{(N)}_N(t))$ and is PD($1/\beta,0$)-distributed, according to the mapping theorem \cite{Kingman1992} and Proposition 10 in \cite{PitmanYor97}.
Furthermore, if $\(Z_1,\dots, Z_N\)$ is a sample without replacement from $(z_1,z_2, \dots)$
with probabilities given by the r.h.s. above, then we can write the family frequencies of 
generation $t+1$ as, for $1\le i\le N$,
$$
\frac{e^{X^{(N)}_i(t+1)}}{\sum_{j=1}^N e^{X^{(N)}_j(t+1)}}
\overset{d}{=}\frac{e^{X^{(N)}_{eq}(t)+Z_i}}{\sum_{j=1}^N e^{X^{(N)}_{eq}(t)+Z_j}}
=\frac{e^{Z_i}}{\sum_{j=1}^N e^{Z_j}}
\overset{d}{=}\frac{\tilde V_{i}^{1/\beta}}{\sum_{j=1}^N \tilde V_j^{1/\beta} } .
$$ 
The r.h.s. above is independent of both $(X^{(N)}_1(t),\dots,X^{(N)}_N(t))$ and $t$, and thus i.i.d over $t$.
\end{proof}
As a straightforward consequence of point i) of Theorem \ref{th:PDmultGenealogy} with $\alpha=1/\beta$, $\theta=0$, and 
$\ppwr=\kappa / \beta$, we obtain the following result.
\begin{theorem}\label{th:genealogiesExpModel}
Consider a $(N,\beta,\kappa)$-EM with $\beta>1$ and $\kappa\in(1/2,1]$.
Then, as $N\to\infty$,
$
c_N\sum_{i=1}^N i^{-\kappa}\to
\frac{\Gamma\(2-\kappa\)}{\kappa \Gamma\(1-\frac{1-\kappa}{\beta}\)
    \Gamma\(1+\frac{1+\kappa}{\beta}\)}
$
and the rescaled genealogical process $\stp{\cpNn_{\floor{t/c_N}}}$ 
converges weakly in the Skorohod topology for $D([0,\infty), \PS{n})$ to the Bolthausen-Sznitman coalescent.
\end{theorem}
Comparing this result with Theorem 1.1 in \cite{CortinesMallein2017} we see that  even in the range $\kappa\in (1/2,1)$,  and in contrast to their model in which other Beta-coalescents appear in the limit, in our case the limit genealogy is always described by the Botlhausen-Sznitman coalescent. Also, in contrast to many
other examples in the literature, our convergence to the Bolthausen-Sznitman coalescent may occur at a time scale other than logarithmic on $N$ depending on the choice of $\kappa$.

\subsection{Martingales associated to the Poisson-Dirichlet distribution}
Before proving Theorem \ref{th:PDmultGenealogy} we develop some general results on the PD($\alpha,\theta$) distribution.
We provide an understanding of the asymptotic behavior of the normalizing sum appearing in \eqref{eq:PDdmpNdef}, mainly 
$$\zeta_{N,\gamma}\coloneqq \sum_{i=1}^N \tilde V_i^\gamma.$$
For this we recall the well-known stick-breaking construction of PD$(\alpha,\theta)$ size-biased picks.
The sequence $(\tilde V_1,\tilde V_2,\dots)$ may be constructed from a collection $(Y_1, Y_2, \dots)$ of independent random variables
where $Y_i$ is Beta($1-\alpha, \theta + i\alpha$) distributed, by setting
$$
\tilde V_1 = Y_1, \text{ and }\tilde V_i = (1-Y_1)\dots(1-Y_{i-1})Y_i.
$$ 
This can be found in Proposition 2 of \cite{PitmanYor97}.

We now write $\zeta_{N,\gamma}$ in terms of the sequence
$(Y_1, Y_2, \dots)$. For this observe that 
\begin{equation}\label{eq:PDprodSN}
\prod_{i=1}^{N}(1-Y_i)^\gamma = e^{-\gamma\mu_N}e^{\gamma S_N}
\end{equation}
where 
$S_N= \mu_N+\sum_{i=1}^N \log(1-Y_i)
$ defines a martingale when $\mu_N= \sum_{i=1}^N -\E\[\log(1-Y_i)\]$. 
We also define the martingale
\begin{align*}
M_{N,\gamma} &\coloneqq \sum_{i=1}^N \(i^{\gamma-1} Y_i^\gamma  - \E\[i^{\gamma-1}Y_i^{\gamma}\]\).
\end{align*}
With this, we have 
\begin{align}
\zeta_{N,\gamma} &= \sum_{i=1}^N Y_i^\gamma e^{-\gamma\mu_{i-1}}e^{\gamma S_{i-1}}\nonumber\\
         &= \sum_{i=1}^N \(M_{i,\gamma} - M_{i-1,\gamma} + \E\[i^{\gamma-1}Y_i^{\gamma}\]\)i^{1-\gamma}e^{-\gamma\mu_{i-1}}
            e^{\gamma S_{i-1}}\nonumber\\
         &= \overline{M}_{N,\gamma} + \Sigma_{N,\gamma}  \label{eq:PDSNident}
\end{align}
where $\overline{M}_{N,\gamma}$ is the martingale
$$
\overline{M}_{N,\gamma} \coloneqq  \sum_{i=1}^N (M_{i,\gamma} - M_{i-1,\gamma}) 
                         i^{1-\gamma}e^{-\gamma\mu_{i-1}}e^{\gamma S_{i-1}}
$$
and $\Sigma_{N,\gamma}$ is the sum
$$
\Sigma_{N,\gamma}\coloneqq \sum_{i=1}^N 
\frac{\Betaf(1+\gamma-\alpha,\theta+i\alpha)}{\Betaf(1-\alpha,\theta+i\alpha)}e^{-\gamma\mu_{i-1}}e^{\gamma S_{i-1}}
$$
since $\E\[Y_i^{\gamma}\]={\Betaf(1+\gamma-\alpha,\theta+i\alpha)}/{\Betaf(1-\alpha,\theta+i\alpha)}$.


We now study the product $\prod_{i=1}^N (1-Y_i)^\gamma$ with the help of the martingale $S_N$ and equation \eqref{eq:PDprodSN}.
The following lemma is an adaptation of Lemma 2.4 of \cite{CortinesMallein2017} into our setting.
Let $\psi(z)=\der{}{z}\log \Gamma(z)$ be the digamma function, and for $a>-1$ let $\EMc_a$ be the 
constant such that $-\log(N) + \sum_{i=1}^N (a + i)^{-1}=\EMc_a + \lo{1}$,
e.g. $\EMc\equiv \EMc_0$ is the Euler-Mascheroni constant.
\begin{lemma}\label{le:PDasympSummands}
 There exists a random variable $S_\infty$ such that
$$
\lim_{N\to\infty} e^{\gamma S_N} = e^{\gamma S_\infty}
$$
almost surely. Furthermore, letting $K_{\alpha,\theta}=\exp\{\psi(\theta+1)-\frac{1}{\alpha}\EMc_{\theta/\alpha}\}$,
the convergence also holds in $\Lp{1}$ whenever $\gamma > -(\theta + \alpha)$ and, in this case,
$$
\E\[e^{\gamma S_\infty}\] =  
 K_{\alpha,\theta}^\gamma \frac{\Gamma(\theta+1)}{\Gamma(\theta+\gamma +1)} 
          \frac{\Gamma\(\frac{\theta+\gamma}{\alpha}+1\)}{\Gamma\(\frac{\theta}{\alpha}+1\)}.
$$
\end{lemma}
\begin{proof}
First recall that if $X$ is Beta$(a,b)$ distributed, then
$\E\[\log(X)\] = \psi(a)-\psi(a+b).$
So
we have $-\E\[\log(1-Y_i)\]=\psi(\theta + (i-1)\alpha + 1)-\psi(\theta + i\alpha)$. Using the
well-known identity $\psi(z+1)=\psi(z)+z^{-1}$,
the bounds $\log(z)-z^{-1}\leq  \psi(z) \leq \log(z+1)-z^{-1}$ for $z>0$ (see Corollary 2.3 in \cite{MuqattashYahdi2006}),
and the definition of $\EMc_{\theta/\alpha}$
, we obtain, as $N\to\infty$, 
\begin{align}\label{eq:muNasmp}
\mu_N &= \psi(\theta+1)-\psi(\theta+N\alpha)+\sum_{i=1}^{N-1} \frac{1}{\theta + i\alpha}\\
      &=  \psi(\theta+1) - \log(N) - \log(\alpha) + \frac{1}{\alpha}\log(N) - \frac{1}{\alpha}\EMc_{\theta/\alpha} + \lo{1}.
      \nonumber  
\end{align}
Thus,  we arrive at
\begin{equation}\label{eq:PDnormAsmp}
e^{\gamma\mu_N} 
\sim \alpha^{-\gamma} K_{\alpha,\theta}^\gamma N^{\gamma(1-\alpha)/\alpha}
\end{equation}
for any $\gamma \in \mathbb{R}$. This, together with \eqref{eq:PDprodSN} and Lemma 2.4 in \cite{CortinesMallein2017}, implies the stated 
a.s. and $\Lp{p}$ convergences for $e^{\gamma S_N}$ to some strictly positive random variable, so that
$S_N$ itself converges to a (finite) r.v. $S_\infty$
(in fact one can also check that $S_N$ is itself an $\Lp{2}$-bounded martingale, we avoid such computations here). 
\end{proof}

We now study the asymptotic behavior of the two terms, $\overline{M}_{N,\gamma}$ and $\Sigma_{N,\gamma}$, that compose 
$\zeta_{N,\gamma}$ in \eqref{eq:PDSNident}.

\begin{lemma}\label{le:PDintMart}
Assume $\gamma > \alpha/2$, then there exists a r.v. 
$\overline{M}_{\infty,\gamma}$ such that
$$
\lim_{N\to\infty} \overline{M}_{N,\gamma} = \overline{M}_{\infty,\gamma} 
$$ 
{ almost surely and in }$\Lp{2}$.
As a straightforward consequence, if $\gamma\leq \alpha$, then
$$
\lim_{N\to\infty} \frac{\overline{M}_{N,\gamma}}{u_N} = 0
$$
{ almost surely and in }$\Lp{2}$.
\end{lemma}
\begin{proof}
We compute, for every 
$\gamma>0$,
\begin{align*}
\E\[(M_{i,\gamma}-M_{i-1,\gamma})^2\]&=\Var{i^{\gamma-1}Y_i^\gamma} \\&= 
\frac{i^{-2}}{\alpha^{2\gamma}}\( \frac{\Gamma(1+2\gamma-\alpha)}{\Gamma(1-\alpha)}-\frac{\Gamma(1+\gamma-\alpha)^2}{\Gamma(1-\alpha)^2} + \bO{i^{-1}}\)
\end{align*}
as $i\to\infty$ so that,
using \eqref{eq:PDnormAsmp}, Lemma \ref{le:PDasympSummands}, and the condition $\gamma>{\alpha}/{2}$
\begin{align*}
&\sum_{i=1}^\infty \Var{\overline{M}_{i+1,\gamma}-\overline{M_{i,\gamma}}}\\
&=\sum_{i=1}^\infty i^{-2{\gamma}/{\alpha}} 
 \E\[e^{2\gamma S_{i}}\] \alpha^{-2\gamma}\(\frac{\Gamma(1+2\gamma-\alpha)}{\Gamma(1-\alpha)}-\frac{\Gamma(1+\gamma-\alpha)^2}{\Gamma(1-\alpha)^2}\)(1+\bO{i^{-1}})\\
&\leq C\E\[e^{2\gamma S_{\infty}}\]\sum_{i=1}^\infty i^{-2\gamma/\alpha}
<\infty
\end{align*}
for some $C>0$.
This implies the a.s. and $\Lp{2}$ convergence of $\overline{M}_{N,\gamma}$ to some r.v. $\overline{M}_{\infty,\gamma}$.
\end{proof}

\begin{lemma}\label{le:PDSumConv}
We have, for $\gamma\leq \alpha$,
\begin{equation*}\label{eq:PDSumConv}
\lim_{N\to\infty} \frac{\Sigma_{N,\gamma}}{u_N} = 
\frac{\Gamma(1+\gamma-\alpha)}{\Gamma(1-\alpha)} 
K_{\alpha,\theta}^{-\gamma}
e^{\gamma S_\infty}
\end{equation*}
almost surely and in $\Lp{1}$.
\end{lemma}
\begin{proof}
Observe that, as $N\to\infty$, 
\begin{equation}\label{eq:PDEYialphaAsmp}
\frac{\Betaf(1+\gamma-\alpha,\theta+N\alpha)}{\Betaf(1-\alpha,\theta+N\alpha)} = 
\frac{\alpha^{-\gamma}\Gamma(1+\gamma-\alpha)}{\Gamma(1-\alpha)}N^{-\gamma}\(1+\bO{N^{-1}}\).
\end{equation}
Then equation \eqref{eq:PDnormAsmp} together with Lemma \ref{le:PDasympSummands} and the Stolz-Césaro 
theorem yield
\begin{align*}
\lim_{N\to\infty} \frac{\Sigma_{N,\gamma}}{u_N}
&{=}\frac{\alpha^{-\gamma}\Gamma(1+\gamma-\alpha)}{\Gamma(1-\alpha)}\lim_{N\to\infty} 
                 \frac{ N^{-\gamma}e^{-\gamma\mu_{N-1}}e^{\gamma S_{N-1}}}{N^{-{\gamma}/{\alpha}}}\\
&{=}\frac{\alpha^{-\gamma}\Gamma(1+\gamma-\alpha)}{\Gamma(1-\alpha)}
                  \alpha^\gamma K_{\alpha,\theta}^{-\gamma}
                  e^{\gamma S_\infty}
\end{align*}
almost surely.
Hence the desired a.s. convergence follows. Similarly, by the same equation \eqref{eq:PDnormAsmp} 
and Lemma \ref{le:PDasympSummands},
$$
\lim_{N\to\infty} \frac{ \E\[\Sigma_{N,\gamma}\]}{u_N}
=\frac{\Gamma(1+\gamma-\alpha)}{\Gamma(1-\alpha)}
   K_{\alpha,\theta}^{-\gamma}
   \E\[e^{\gamma S_\infty}\]
$$
so that Scheffe's lemma gives the corresponding $\Lp{1}$ convergence.
\end{proof}
Thanks to the previous lemmas, we obtain the asymptotic behavior of $\zeta_{N,\gamma}$.
\begin{proposition}\label{prop:PDnormsumAsmp}
If $\alpha/2< \gamma \leq \alpha$, then 
\begin{equation*}
\lim_{N\to\infty} \frac{\zeta_{N,\gamma}}{u_N} = 
\frac{\Gamma(1+\gamma-\alpha)}{\Gamma(1-\alpha)} 
K_{\alpha,\theta}^{-\gamma}
e^{\gamma S_\infty}
\end{equation*}
almost surely and in $\Lp{1}$.
Moreover, if $\eta\gamma<\theta+\alpha$, then 
\begin{equation}\label{eq:asmpZetaEta}
\lim_{N\to\infty} \(\frac{\zeta_{N,\gamma}}{u_N}\)^{-\eta} =
\(\frac{\Gamma(1+\gamma-\alpha)}{\Gamma(1-\alpha)}K_{\alpha,\theta}^{-\gamma} 
e^{\gamma S_\infty}\)^{-\eta}
\end{equation}
almost surely and in $\Lp{1}$. \par 
On the other hand, if $\gamma>\alpha$, then there exists a r.v. $\zeta_{\infty,\gamma}$ such that
\begin{equation}\label{gammay}
\lim_{N\to\infty} \zeta_{N,\gamma}  = \zeta_{\infty,\gamma} 
\end{equation}
almost surely and in $\Lp{p}$ for $p\geq1$.
\end{proposition}
\begin{proof}
We first assume $\gamma\leq\alpha$. The a.s. and $\Lp{1}$ convergences in this case follow directly from \eqref{eq:PDSNident} and Lemmas \ref{le:PDintMart} 
and \ref{le:PDSumConv}. The second convergences will follow once we prove 
\begin{equation}\label{eq:PDnormSuminL1}
\limsup_{N\geq0}\E\[\(\frac{u_N}{\zeta_{N,\gamma}}\)^{\eta}\]<\infty
\end{equation}
together with an application of dominated convergence and Scheffe's lemma. To upperbound the expectations in \eqref{eq:PDnormSuminL1} observe that, 
if $\E_\Omega$ is the expectation operator on the 
probability space $$\(\Omega=\{f(1),\dots,f(N)\}, \frac{\sum_{i=1}^N \delta_{f(i)}i^{-{\gamma}/{\alpha}}}{u_N}\),$$
where $f(i)\coloneqq Y_i^\gamma e^{-\gamma\mu_{i-1}}e^{\gamma S_{i-1}}i^{{\gamma}/{\alpha}}$,
then, by Jensen's inequality we have
\begin{align*}
{\zeta_{N,\gamma}^{-\eta}}
&=\exp\lbr -\eta\log\(\sum_{i=1}^N Y_i^\gamma e^{-\gamma\mu_{i-1}}e^{\gamma S_{i-1}}i^{{\gamma}/{\alpha}}i^{-{\gamma}/{\alpha}}\)\rbr\\
&=e^{-\eta \log\(u_N\)} 
  e^{ -\eta\log\(\E_\Omega\[x\]\) }
\\
&\leq e^{-\eta \log\(u_N\)} 
  e^{- \eta\E_\Omega\[\log(x)\] }\\
&= 
    \(u_N\)^{-\eta}
      \exp\lbr- \frac{\eta\gamma}{u_N}\sum_{i=1}^N \frac{\log(Y_i)-\mu_{i-1}+S_{i-1}+\frac{1}{\alpha}\log(i)}{i^{{\gamma}/{\alpha}}}\rbr.
\end{align*}
Furthermore, by Lemma \ref{le:PDasympSummands} and the condition $\eta\gamma<\alpha+\theta$ we have, for 
$\eps>0$ small enough and using that $\(e^{-\eta\gamma (1+\eps)S_N}\)_{N\le1}$ is a submartingale,
$$
\exp\lbr-\eta\gamma(1+\eps)  \frac{\sum_{i=1}^N S_{i-1}i^{-{\gamma}/{\alpha}}}{u_N} \rbr\leq \sup_{N\in\N} e^{-\eta\gamma (1+\eps) S_{N}}\in \Lp{1}.
$$
Thus, plugging these estimates and using H\"older's inequality ($p=1+\eps$), 
\begin{align}\label{eq:zetaNasmp1}
&\E\[\(\frac{u_N}{\zeta_{N,\gamma}}\)^{\eta}\]\notag\\
&\leq \norm{\sup_{N\in\N} e^{-\eta\gamma (1+\eps) S_{N}}}_{{1}}
\E\[\exp\lbr-\frac{1+\eps}{\eps} \frac{\eta\gamma}{u_N}
                    \sum_{i=1}^N \frac{\log(Y_i)-\mu_{i-1}+\frac{1}{\alpha}\log(i)}{i^{{\gamma}/{\alpha}}} \rbr\]^{\frac{\eps}{1+\eps}}.
\end{align}
We now compute,
for $N$ large enough
to ensure $\frac{1+\eps}{\eps} \frac{\eta\gamma}{u_N}<1-\alpha$, 
and using an easy consequence
of the mean value theorem on the function $z\to\log\Gamma(z)$, 
\begin{align*}
&\E\[\exp\lbr -\eta\gamma\frac{1+\eps}{\eps} \frac{\log(Y_i)i^{-{\gamma}/{\alpha}}}{u_N}
          \rbr\]=\E\[Y_i^{-\eta\gamma \frac{1+\eps}{\eps} 
         \frac{i^{-{\gamma}/{\alpha}}}{u_N}}\]\\
&= \exp\lbr \log\Gamma\(1-\alpha-\eta\gamma \frac{1+\eps}{\eps}\frac{i^{-{\gamma}/{\alpha}}}{u_N}\)
               -\log\Gamma\(1-\alpha\) \rbr \\
&\times \exp\lbr 
           \log\Gamma(1+\theta+(i-1)\alpha)-
           \log\Gamma\(1+\theta+(i-1)\alpha - \eta\gamma \frac{1+\eps}{\eps}\frac{i^{-{\gamma}/{\alpha}}}{u_N}\)\rbr\\
&\leq \exp \lbr \eta\gamma \frac{1+\eps}{\eps}\frac{i^{-{\gamma}/{\alpha}}}{u_N}\psi\(1+\theta+(i-1)\alpha\) 
       \rbr.
\end{align*}
Equation \eqref{eq:muNasmp} and the identity 
$\psi(1+\theta + (i-1)\alpha)=\psi(\theta + (i-1)\alpha) + \frac{1}{\theta + (i-1)\alpha}$ 
yield, as $i\to\infty$,
$$
\psi\(1+\theta+(i-1)\alpha\)-\mu_{i-1}+\frac{1}{\alpha}\log(i)
=\frac{1}{\theta + (i-1)\alpha} -\psi(\theta+1) + \frac{1}{\alpha}\EMc_{\theta/\alpha} + \lo{1}.
$$
Thus we obtain, for all $i\leq N$, that there exists $C'>0$ such that
$$ \E\[\exp\lbr-\frac{\eta\gamma}{u_N}\frac{1+\eps}{\eps} 
 {i^{-{\gamma}/{\alpha}}(\log(Y_i)-\mu_{i-1}+\frac{1}{\alpha}\log(i))}
 \rbr\]^{\frac{\eps}{1+\eps}}
\leq \exp\lbr C'\frac{\eta\gamma}{u_N} {i^{-{\gamma}/{\alpha}}} \rbr.
$$
Taking the product over $i$, and plugging in \eqref{eq:zetaNasmp1}, we conclude, for all $N\in\N$,
$$
\E\[\(\frac{u_N}{\sum_{i=1}^N \tilde V_i^\gamma}\)^{\eta}\]
<\exp\{C'\eta\gamma \},
$$
which entails \eqref{eq:PDnormSuminL1}.\par
We now prove \eqref{gammay}. Let us assume $\gamma > \alpha$ and bound the first line of \eqref{eq:PDSNident} in $\Lp{p}$. 
Using equation \eqref{eq:PDnormAsmp} to approximate $e^{-\gamma \mu_i}$, the fact that
$
\E\[Y_i^{p\gamma}\]=\frac{\Betaf(1+p\gamma-\alpha,\theta+i\alpha)}{\Betaf(1-\alpha,\theta+i\alpha)}
$ together with \eqref{eq:PDEYialphaAsmp}, and also using Lemma \ref{le:PDasympSummands} to upperbound 
$\norm{e^{\gamma S_i}}_p$, we obtain that for some constant 
$C>0$ and any $p\geq1$,
\begin{align*}
\norm{\zeta_{N,\gamma}}_p&\leq \sum_{i=1}^N \norm{Y_i^\gamma}_p e^{-\gamma\mu_{i-1}}\norm{e^{\gamma S_{i-1}}}_p\\
&\leq  C \norm{e^{\gamma S_\infty}}_p \sum_{i=1}^\infty i^{-\gamma} i^{\gamma-{\gamma}/{\alpha}} \(1+\lo{1}\)
<\infty.
\end{align*}
In particular the increasing sequence $\(\zeta_{N,\gamma}^p\)_{N\in\N }$ is a.s. bounded and thus convergent. 
The convergence also holds in $\Lp{p}$ by dominated convergence which yields 
$\norm{\zeta_{N,\gamma}}_{{p}} \to \norm{\zeta_{\infty,\gamma}}_{{p}}$.
\end{proof}

\subsection{Proof of Theorem \ref{th:PDmultGenealogy}}
In this section we gather the results obtained so far in order to prove Theorem \ref{th:PDmultGenealogy}. Before we do so, 
we recall yet another well-known fact on the Poisson-Dirichlet distribution, mainly that
 removing the element $\tilde V_1$ from $(\tilde V_1, \tilde V_2, \dots)$ and renormalizing the resulting sequence leads, in
distribution, to a change of  parameters.
\begin{lemma}[Poisson-Dirichlet change of parameter]\label{le:PDsummary}
Let $(\tilde V_1,\tilde V_2,\dots)$ be a size-biased pick from a PD$(\alpha,\theta)$ partition. Then 
the sequence $\(\frac{\tilde V_2}{1-\tilde V_1}, \frac{\tilde V_3}{1-\tilde V_1}, \dots\)$  is distributed as a size-biased pick from a 
PD($\alpha, \theta + \alpha$) partition.
\end{lemma}
The martingale properties obtained in the previous section entail that $c_N$ can be compared to $\E\[\(\rpvN_1\)^2\]$.
\begin{lemma}\label{lastlem}
Assume $\alpha/2<\ppwr\leq \alpha$. Then there exists a constant $B\equiv B_{\alpha,\theta,\ppwr}$ such that 
\begin{align*}
c_N\ge Bu_N^{-2}.
\end{align*}
Moreover, when  $\theta\leq 0$
$$c_N\le\E\[\(\rpvN_1\)^2\] +\lo{u_N^{-\frac{\alpha+\theta}{\alpha}}},$$
whereas 
$$c_N\le\E\[\(\rpvN_1\)^2\] + \bO{u_N^{-2}}$$
whenever $\theta>0$. 
\end{lemma}
\begin{proof}
On the one hand, choose $p>3$. By the reverse Hölder's inequality we have
\begin{align*}
c_N &= \E_{\alpha,\theta}\[\frac{\zeta_{N,2\gamma}}{\zeta_{N,\gamma}^2}\] \geq  
\E_{\alpha,\theta}\[\frac{\zeta_{N,2\alpha}}{\zeta_{N,\gamma}^2}\]
\geq\E_{\alpha,\theta} \[{\zeta_{N,2\alpha}^\frac1p}\]^p\E_{\alpha,\theta}\[\zeta_{N,\gamma}^{\frac{2}{p-1}}\]^{1-p},
\end{align*}
where $\E_{\alpha,\theta}$ refers to expectation with respect to the PD$(\alpha,\theta)$ distribution. 
Thanks to Proposition \ref{prop:PDnormsumAsmp} the first term in the r.h.s. above is of order $\bO{1}$
as $N\to\infty$; whereas the second term is of order $\bO{u_N^{-\frac{2}{p-1}(1-p)}}$
(using $2\alpha>\alpha$ for the first, and $\frac{2\gamma}{(p-1)}>0>-(\alpha+\theta)$ for the second).
We thus conclude that
\begin{align*}
c_N
&\geq 
      Bu_N^{-2}
\end{align*}
for some $B>0$. 

On the other hand, by  Lemma \ref{le:PDsummary} in the second equality below, we have for any $\delta>0$,
\begin{align}\label{eq:etas2deltabound}
&\E_{\alpha,\theta}\[\sum_{i=2}^N \(\rpvN_i\)^\delta\]\nonumber
\\&= \E_{\alpha,\theta}\[
\frac{\sum_{i=2}^N\(\frac{\tilde V_{i}}{ 1-\tilde V_1}\)^{\delta\ppwr}(1-\tilde V_1)^{\delta\ppwr}}
     { \( \tilde V_1^\ppwr + (1-\tilde V_1)^\ppwr \sum_{j=2}^N \(\frac{\tilde V_j}{1-\tilde V_1}\)^\ppwr\)^\delta}
\] \nonumber\\
&=\E_{\alpha,\theta+\alpha}\[\int_0^1 \frac{(1-y)^{\delta\gamma}\zeta_{N-1,2\gamma}}{\(y^\gamma+(1-y)^\gamma\zeta_{N-1,\gamma}\)^\delta}
\frac{y^{-\alpha}(1-y)^{\theta+\alpha-1}}{\Betaf(1-\alpha,\theta+\alpha)} dy\]\nonumber\\
&=\E_{\alpha,\theta+\alpha}\[
\frac{\zeta_{N-1,\delta\gamma}}{\zeta_{N-1,\gamma}^\delta} 
\int_0^1 \frac{(1-y)^{\delta\gamma}}{\(y^\gamma/\zeta_{N-1,\gamma}+(1-y)^\gamma\)^\delta}
\frac{y^{-\alpha}(1-y)^{\theta+\alpha-1}}{\Betaf(1-\alpha,\theta+\alpha)} dy \]\nonumber\\
&\leq \E_{\alpha,\theta+\alpha}\[\frac{\zeta_{N-1,\delta\gamma}}{\zeta_{N-1,\gamma}^\delta}\].
\end{align}
If we set $\frac{\alpha}{\gamma} <\delta<\frac{2\alpha+\theta}{\gamma}$  and $p=1+\eps,q=\frac{p}{p-1}$ with
$\eps>0$ small enough to ensure $p\delta\gamma<2\alpha+\theta$, we can use Proposition \ref{prop:PDnormsumAsmp}.
By \eqref{eq:asmpZetaEta}, with $\eta=p\delta$, we get
$\E_{\alpha,\theta+\alpha}\[\zeta_{N-1,\gamma}^{-p\delta}\]^{\frac1p}=\bO{\(u_N\)^{-\delta}}$ . 
By \eqref{gammay}, and observing that $\delta\gamma>\alpha$, we get
$\E_{\alpha,\theta+\alpha}\[\zeta_{N-1,\delta\gamma}^q\]^{\frac1q}=\bO{1}$. 
Thus, since $\gamma>\alpha/2$  we may choose $\delta\in[0,2]\cap\(\frac{\alpha}{\gamma},\frac{2\alpha+\theta}{\gamma}\)\neq \emptyset$
 and use Hölder's inequality, with $p$ and $q$ as before,
in the last line above to obtain
$$\E_{\alpha,\theta}\[\sum_{i=2}^N \(\rpvN_i\)^2\]
\leq \E_{\alpha,\theta}\[\sum_{i=2}^N \(\rpvN_i\)^\delta\]=\bO{u_N^{-\delta}}.$$ 
In particular, since $\gamma\leq\alpha$, for $\theta\leq0$ we may conclude $\E_{\alpha,\theta}\[\sum_{i=2}^N \(\rpvN_i\)^2\]=\lo{u_N^{-\frac{\alpha+\theta}{\alpha}}}$, whereas for $\theta>0$ we obtain
$\E_{\alpha,\theta}\[\sum_{i=2}^N \(\rpvN_i\)^2\]=\bO{u_N^{-2}}$.
\end{proof}
We now prove the stated convergence of the genealogy.
\begin{proof}[{Proof of Theorem \ref{th:PDmultGenealogy}}]
We first derive asymptotics for $\E_{\alpha,\theta}\[\(\rpvN_1\)^\delta\]$ where $\delta>0$.
By  Lemma \ref{le:PDsummary}
we have
\begin{align*}
&\E_{\alpha,\theta}\[\(\rpvN_1\)^\delta\]\\
&=\E_{\alpha,\theta+\alpha}
  \[ \int_0^1 \frac{y^{\gamma \delta}}{\(y^\gamma+(1-y)^\gamma \zeta_{N-1,\gamma}\)^\delta} y^{-\alpha}(1-y)^{\theta+\alpha-1} 
   \frac{dy}{\Betaf(1-\alpha,\theta+\alpha)}\],
\end{align*}
where, making the change of variable $u=(1-y)^\gamma\zeta_{N-1,\gamma}$,  we obtain
\begin{align}\label{eq:PDcoaCondii}
\hspace{-50pt}
&\E_{\alpha,\theta}\[\(\rpvN_1\)^\delta\]
=\frac{1}{\gamma\Betaf(1-\alpha,\theta+\alpha)}\times\nonumber\\
&\quad  \E_{\alpha,\theta+\alpha}\[
        \frac{ 1 }
           {\zeta_{N-1,\gamma}^{1+\frac\theta\alpha}} 
        \int_0^{\zeta_{N-1,\gamma}} \frac{\(1-\frac{u^{\frac1\gamma}}{\zeta_{N-1,\gamma}^{\frac1\gamma}}\)^{\gamma \delta}}
           {\(\(1-\frac{u^{\frac1\gamma}}{\zeta_{N-1,\gamma}^{\frac1\gamma}}\)^{\gamma}+ u\)^\delta}
           \(1-\frac{u^{\frac1\gamma}}{\zeta_{N-1,\gamma}^{\frac1\gamma}}\)^{-\alpha} u^{\frac\theta\alpha}   du \].
\end{align}
Denote by $I_N\equiv I_{N,\delta}$ the integral inside the expectation above. Choose $0<\eps<1$. Then $I_N$ can be bounded on the set $\lbr\frac{u^{1/\gamma}}{\zeta_{N-1,\gamma}^{1/\gamma}}\leq\eps\rbr$ by
\begin{align*}
 \int_0^\infty \frac{u^{\frac\theta\alpha}}{((1-\eps)^\gamma + u)^\delta}du <\infty 
\end{align*}
 whenever $\delta>1+\frac{\theta}{\alpha}$. 
On the other hand, on the set $\lbr\frac{u^{1/\gamma}}{\zeta_{N-1,\gamma}^{1/\gamma}}>\eps\rbr$ the integral $I_N$ 
can be bounded by
$$
 \int_{\eps^\gamma\zeta_{N-1,\gamma}}^{\zeta_{N-1,\gamma}}
 (1-\eps)^{2\gamma  - \alpha}u^{\frac\theta\alpha-\delta}du  \leq  C \zeta_{N-1, \gamma}^{1+\frac{\theta}{\alpha}-\delta}
$$
for some $C>0$. 
These two bounds, together with 
the fact that $\zeta_{N-1,\gamma}\to\infty$ as $N\to\infty$  (from Proposition \ref{prop:PDnormsumAsmp}), 
allow us to use (pathwise) dominated convergence to obtain, for $\delta>1+\frac{\theta}{\alpha}$ and as $N\to\infty$,
$$
I_N
{\to} \int_0^\infty \frac{u^{\frac\theta\alpha}}{(1+u)^\delta}du
$$
almost surely. On the other hand, substituting these bounds
 into the expectation in \eqref{eq:PDcoaCondii} we obtain  
$\E_{\alpha,\theta}\[\(\rpvN_1\)^\delta\]=
\bO{ \E_{\alpha,\alpha+\theta}\[\zeta_{N-1,\gamma}^{1+\frac{\theta}{\alpha}}\]} + \bO{\E_{\alpha,\alpha+\theta}\[\zeta_{N-1,\gamma}^{-\delta}\]}$ for $\delta>1+\frac{\theta}{\alpha}$.
By \eqref{eq:asmpZetaEta} in Proposition \ref{prop:PDnormsumAsmp} together with the fact that 
$\frac{\alpha+\theta}{\alpha}\gamma<2\alpha+\theta$, the first term is of 
order $\bO{u_N^{-1-\frac{\theta}{\alpha}}}$. If furthermore 
$\delta<\frac{2\alpha+\theta}{\gamma}$ then, again by \eqref{eq:asmpZetaEta}, the
second term is of order $\bO{ u_N^{-\delta}}$. Since the random variable 
$\(\rpvN_1\)^\delta$ is decreasing
in $\delta$, the condition $\delta<\frac{2\alpha+\theta}{\gamma}$ was no restriction, and we may conclude 
$\E_{\alpha,\theta}\[\(\rpvN_1\)^\delta\]
=\bO{u_N^{-1-\frac{\theta}{\alpha}}}$
for all $\delta>1+\frac{\theta}{\alpha}$.
Thus we may apply dominated convergence in \eqref{eq:PDcoaCondii} which, together with Proposition \ref{prop:PDnormsumAsmp},
yield for $\delta>1+\frac{\theta}{\alpha}$,
\begin{align}\label{eq:PDeta1Asmp}
&\lim_{N\to\infty} u_N^{1+\frac{\theta}{\alpha}}\E_{\alpha,\theta}\[\(\rpvN_1\)^\delta\]\\
&=
  \E_{\alpha,\alpha+\theta}\[ \(\frac{\Gamma(1+\gamma-\alpha)}{\Gamma(1-\alpha)}K_{\alpha,\theta}^{-\gamma}e^{\gamma S_\infty}\)^{-1-\frac{\theta}{\alpha}} \]
  \frac{\gamma^{-1}}{\Betaf(1-\alpha,\alpha+\theta)}
       \int_0^\infty \frac{u^{\frac\theta\alpha}}{(1+u)^\delta}du.\nonumber
\end{align}
We now turn to the case $i)$.
Given Lemma \ref{lastlem}, and the condition $\theta\in (-\alpha,\alpha)$, it
 remains to prove the condition \eqref{critsimple} in Proposition \ref{prop:multSPlambda}. This is, in our context, 
\begin{equation}\label{eq:PDbetaii}
\lim_{N\to\infty}  u_N^{1+\frac{\theta}{\alpha}}\E_{\alpha,\theta}\[\(\rpvN_1\)^b\]= \ell_{\alpha,\theta,\gamma}\int_0^1 p^{b-2}\Lambda(dp) 
\end{equation}
 where $\Lambda(dp)$ is the
coagulation measure of the Beta$(1+\frac{\theta}{\alpha}, 1-\frac{\theta}{\alpha})$-coalescent.
When $\theta\in\(-\alpha,\alpha\)$ we have $1+\frac{\theta}{\alpha}<2$. 
Thus \eqref{eq:PDeta1Asmp} holds for all $\delta=b\geq2.$
By means of the change of variable $x=(1+u)^{-1}$ we may rewrite the integral in \eqref{eq:PDeta1Asmp} with $\delta=b$ as
$$
       \int_0^\infty \frac{u^{\frac\theta\alpha}}{(1+u)^b}du
 = \int_0^1 x^{b-2}x^{1-\frac\theta\alpha-1}(1-x)^{1+\frac\theta\alpha-1} dx.
$$
On the other hand, by Lemma \ref{le:PDasympSummands}, 
\begin{align*}
&\E_{\alpha,\alpha+\theta}\[\(\frac{\Gamma(1+\gamma-\alpha)}{\Gamma(1-\alpha)}K_{\alpha,\theta}^{-\gamma}e^{\gamma S_\infty}\)^{-1-\frac{\theta}{\alpha}} \]
  \frac{\gamma^{-1}}{\Betaf(1-\alpha,\alpha+\theta)}\\
&=\gamma^{-1}\(\frac{\Gamma(1-\alpha)}{\Gamma(1+\gamma-\alpha)}\)^{1+\frac{\theta}{\alpha}} 
    \frac{\Gamma(\alpha+\theta+1)}{\Gamma\(\alpha+\theta-\frac{\gamma}{\alpha}(\alpha+\theta) +1\)} 
       \notag\\
&\times   \frac{\Gamma\(\frac{1}{\alpha}(\alpha+\theta-\frac{\gamma}{\alpha}(\alpha+\theta))+1\)}{\Gamma\(\frac{\alpha+\theta}{\alpha}+1\)}
\frac{\Gamma(1+\theta)}{\Gamma(1-\alpha)\Gamma(\alpha+\theta)}\\
&= \frac{\alpha}{\gamma}\frac{\Gamma(1-\alpha)^{\frac{\theta}{\alpha}}}{\Gamma(1+\gamma-\alpha)^{1+\frac{\theta}{\alpha}}}    \frac{\Gamma\(\frac{\alpha+\theta}{\alpha}(1-\frac{\gamma}{\alpha})+1\)}
         {\Gamma\(\(\alpha+\theta\)\(1-\frac{\gamma}{\alpha}\) +1\)}
   \frac{\Gamma(1+\theta) \Gamma\(1-\frac{\theta}{\alpha}\)}
        {\Betaf\(1-\frac{\theta}{\alpha},1+\frac{\theta}{\alpha}\)}.
\end{align*}
Substituting both in \eqref{eq:PDeta1Asmp} we obtain \eqref{eq:PDbetaii}.\par

Let us now turn to case $ii)$;      
we will prove the conditions of Proposition \ref{prop:multiKingman2}. For the first condition, recall that, by
 Lemma \ref{lastlem}, $c_N\geq Cu_N^{-2}$ for some $C>0$. Then, since
 $\frac{\alpha}{\gamma}<2$ and $\frac{2\alpha+\theta}{\gamma}\geq 3$;  
the argument after equation \eqref{eq:etas2deltabound} yields for $2<\delta<3$ (and any $\theta\geq \alpha$),
$$
\E_{\alpha,\theta}\[\sum_{i=2}^N \(\rpvN_i\)^3\]=\bO{u_N^{-\delta}}
=\lo{c_N}.
$$\par                                    
It remains to prove the second condition of  Proposition \ref{prop:multiKingman2}. When $\theta>\alpha$ we also have, choosing $\delta>1+\frac{\theta}{\alpha}\geq2$
 in \eqref{eq:PDeta1Asmp}, that
$\E_{\alpha,\theta}[\(\rpvN_1\)^\delta]=\bO{u_N^{-1-\frac{\theta}{\alpha}}}=\lo{c_N}$. 
The convergence of the genealogy then follows in this case.\par
The special case when $\theta=\alpha$ needs to be treated separately. Observe that for $\delta= 1+\frac{\theta}{\alpha}=2$ 
we have the following bounds on $I_N$  (for some $C_1,C_2\in\mathbb{R}$)
$$
C_1+\int_2^{\zeta_{N-1,\gamma}}\frac{\(1-\frac{u}{\zeta_{N-1,\gamma}}\)^{2}}{(\frac{1}{u}+1)^2}u^{-1}du
\leq I_N \leq C_2+\int_2^{\zeta_{N-1,\gamma}}u^{-1}du.
$$
These, together with the fact that (by an application of the Mean Value Theorem) 
$\abs{1-\(1-\frac{u}{\zeta_{N-1,\gamma}}\)^{2}}
\leq u/\zeta_{N-1,\gamma}$, and that 
$\lim_{x\to\infty}\frac{1}{\log(x)}\int_2^x\frac{1}{(1/u+1)^2}u^{-1}du=\lim_{x\to\infty}\frac{1}{\log(x)}\int_2^xu^{-1}du$,
yield using dominated convergence,
$$
\lim_{N\to\infty}\frac{I_N}{\log(\zeta_{N-1,\gamma}\vee2)}{=}
\lim_{N\to\infty} \frac{1}{\log(\zeta_{N-1,\gamma}\vee 2)}\int_2^{\zeta_{N-1,\gamma}}u^{-1}du
=1
$$
almost surely. 
Since $I_N\leq C_2+\log(\zeta_{N-1,\gamma}\vee 2)$, the above convergence also holds in $\Lp{p}$,$p>0$.
On the other hand, by  Proposition \ref{prop:PDnormsumAsmp}, the limit
$
\lim_{N\to\infty}\log\((\zeta_{N-1,\gamma}\vee2) / \break u_N^{1+\frac{\theta}{\alpha}}\)
$
exists almost surely and in $\Lp{p},p>0$. 
Thus
\begin{align*}
\lim_{N\to\infty}\frac{\log\(\zeta_{N-1,\gamma}\vee2\)}{(1+\frac{\theta}{\alpha})\log\(u_N\)}
=1.
\end{align*}
This implies, multiplying and dividing the l.h.s. below by $\log(\zeta_{N-1,\gamma}\vee 2)$,
\begin{align*}\lim_{N\to\infty}\frac{I_N}{(1+\frac{\theta}{\alpha})\log\(u_N\)}
&=\lim_{N\to\infty}\log\((\zeta_{N-1,\gamma}\vee2)/u_N^{1+\frac{\theta}{\alpha}}\)\\
&=\log\(\frac{\Gamma(1+\gamma-\alpha)}{\Gamma(1-\alpha)}K^{-\gamma}_{\alpha,\theta}e^{\gamma S_\infty}\)
\end{align*}
almost surely and in $\Lp{p}$, $p>0$. Plugging into \eqref{eq:PDcoaCondii} we then obtain, choosing $p$ close enough to 1
so that $p(1+\frac{\theta}{\alpha})=2p<\frac{2\alpha+\theta}{\gamma}$ and $q=\frac{p}{p-1}$,
\begin{align*}
\E_{\alpha,\theta}\[\(\rpvN_1\)^2\]
&\leq \E_{\alpha,\alpha+\theta}\[\zeta_{N-1,\gamma}^{-2p}\]^{1/p} \E_{\alpha,\alpha+\theta}\[I_N^q\]^{1/q}\\
&=\bO{u_N^{-2}\log\(u_N\)}.
\end{align*}
By dominated convergence in  \eqref{eq:PDcoaCondii}, the latter implies
$$
\lim_{N\to\infty}\frac{u_N^{2}}{\log\(u_N\)}
\frac{\E_{\alpha,\alpha+\theta}\[\(\rpvN_1\)^2\]}{1+\frac{\theta}{\alpha}}
=\E\[\frac{\log\(\frac{\Gamma(1+\gamma-\alpha)}{\Gamma(1-\alpha)}K_{\alpha,\theta}^{-\gamma}e^{\gamma S_\infty}\)}{\(\frac{\Gamma(1+\gamma-\alpha)}{\Gamma(1-\alpha)}K_{\alpha,\theta}^{-\gamma}e^{\gamma S_\infty}\)^{2}}\].
$$
By Lemma \ref{lastlem} we thus have $u_N^{-2}=\lo{\E_{\alpha,\theta}\[\(\rpvN_1\)^2\]}=\lo{c_N}$. Then, choosing $\delta>1+\frac{\theta}{\alpha}$ in
\eqref{eq:PDeta1Asmp} we obtain $\E\[\(\rpvN_1\)^\delta\]=\lo{c_N}$. The proof is completed by a second application
of Proposition \ref{prop:multiKingman2}. 
\end{proof}

\section*{Statements and Declarations}
The authors have no relevant financial or non-financial interests to disclose. 
\section*{Data Availability}
Data sharing not applicable to this article as no datasets were generated or analysed during the current study.

\section*{Acknowledgments}

This project was partially supported by Universidad Nacional Autónoma de México - Programa de Apoyo de Proyectos de Investigación e Inovación Tecnológica grant IN104722.\par
This work was partially supported by the ANR LabEx CIMI (grant ANR-11-LABX-0040) within the French State Programme “Investissements d’Avenir.”\par
The authors also thank the Hausdorff Institute for Mathematics, where they could spend some very fruitful research time during the Junior Trimester Program ``Stochastic Modelling in the Life Science: From Evolution to Medicine" funded by the Deutsche Forschungsgemeinschaft (DFG, German Research Foundation) under Germany's Excellence Strategy – EXC-2047/1 – 390685813.

\printbibliography

\end{document}